\documentclass[preprint]{elsarticle}

\usepackage{amsthm}
\usepackage{amssymb}
\usepackage{amsfonts}
\usepackage{amsxtra}
\usepackage{mathrsfs}
\usepackage[all]{xy}
\usepackage{amsmath}
\usepackage{color}
\usepackage{graphicx}
\usepackage{enumerate}

\newcommand{\p}{\medskip \noindent}

\newcommand{\PP}{\mathscr{P}}

\newcommand{\cc}{{^\frown}}

\newcommand{\till}{{\upharpoonright}}
\newcommand{\thru}{\uparrow}

\newcommand{\gen}{{\rm gen}}

\newcommand{\G}{\boldsymbol{\Gamma}}
\newcommand{\SIGMA}{\boldsymbol{\Sigma}}
\newcommand{\DELTA}{\boldsymbol{\Delta}}
\newcommand{\PI}{\boldsymbol{\Pi}}

\newcommand{\Split}{{\rm Split}}

\newcommand{\ww}{\omega^\omega}

\newcommand{\dw}{2^\omega}

\newcommand{\kk}{\kappa^\kappa}
\newcommand{\klk}{\kappa^{<\kappa}}
\newcommand{\dk}{2^\kappa}
\newcommand{\dlk}{2^{<\kappa}}
\newcommand{\kuk}{\kappa_{\uparrow}^{\kappa}}
\newcommand{\kulk}{\kappa_{\uparrow}^{{<}\kappa}}

\newcommand{\sig}{\boldsymbol{{\sigma}}}
\newcommand{\quine}[1]{\ulcorner{#1}\urcorner}
\newcommand{\last}{{\rm last}}
\newcommand{\height}{{\rm height}}
\newcommand{\Term}{{\rm Term}}

\newcommand{\stem}{{\rm stem}}
\newcommand{\Succ}{{\rm Succ}}

\newcommand{\IP}{\mathbb{P}}
\newcommand{\IQ}{\mathbb{Q}}

\newcommand{\ran}{{\rm ran}}

\newcommand{\s}[1]{\medskip \noindent \textbf{#1}}

\newcommand{\N}{\mathcal{N}}

\newcommand{\IS}{\mathbb{S}}
\newcommand{\IL}{\mathbb{L}}

\newcommand{\IM}{\mathbb{M}}
\newcommand{\IR}{\mathbb{R}}

\newcommand{\IV}{\mathbb{V}}
\newcommand{\IC}{\mathbb{C}}

\newtheorem{Thm}{Theorem}[section]

\newtheorem{Lem}[Thm]{Lemma}
\newtheorem{SubLem}[Thm]{Sublemma}
\newtheorem{Cor}[Thm]{Corollary}

\newtheorem{Question}[Thm]{Question}
\newtheorem{Fact}[Thm]{Fact}

\theoremstyle{definition}
\newtheorem{Def}[Thm]{Definition}
\newtheorem{Example}[Thm]{Example}
\newtheorem{Notation}[Thm]{Notation}
\newtheorem{Remark}[Thm]{Remark}

\newtheorem{Observation}[Thm]{Observation}

\newcommand{\I}{\mathcal{I}}

\begin{document}

\begin{frontmatter}

\title{Regularity Properties on the Generalized Reals.}

\author{Sy David Friedman\fnref{sy}}
\ead{sdf@logic.univie.ac.at}
\address{Kurt G\"odel Research Center for Mathematical Logic, University of Vienna, \\ W\"ahringer Stra\ss e 25, 1090 Wien, Austria}

\author{Yurii Khomskii\fnref{yurii}\corref{a}}
\ead{yurii@deds.nl}
\address{Kurt G\"odel Research Center for Mathematical Logic, University of Vienna, \\ W\"ahringer Stra\ss e 25, 1090 Wien, Austria}
\cortext[a]{Corresponding author}

\author{Vadim Kulikov\fnref{vadim}}
\ead{vadim.kulikov@iki.fi}
\address{Kurt G\"odel Research Center for Mathematical Logic, University of Vienna, \\ W\"ahringer Stra\ss e 25, 1090 Wien, Austria}

\fntext[sy]{Supported by the Austrian Science Fund (FWF) under project numbers P23316 and P24654.}
\fntext[yurii]{Supported by the Austrian Science Fund (FWF) under project number P23316.}
\fntext[vadim]{Supported by the Austrian Science Fund (FWF) under project number P24654.}

\begin{abstract} We investigate regularity properties derived from tree-like forcing notions in the setting of ``generalized descriptive set theory'', i.e., descriptive set theory on $\kk$ and $\dk$, for regular uncountable cardinals $\kappa$.
\end{abstract}


\end{frontmatter}

\section{Introduction}

\emph{Generalized Descriptive Set Theory} is an area of research  dealing with generalizations of  classical descriptive set theory on the Baire space $\ww$ and Cantor space $\dw$, to the \emph{generalized Baire space} $\kk$ and the \emph{generalized Cantor space} $\dk$, where $\kappa$ is an uncountable regular cardinal satisfying $\kappa^{<\kappa} = \kappa$. Some of the earlier papers dealing with descriptive set theory on $(\omega_1)^{\omega_1}$ were motivated by model-theoretic concerns, see e.g. \cite{VaananenTrees} and \cite[Chapter 9.6]{VaananenModelsGames}. More recently, generalized descriptive set theory became a field of interest in itself, with various aspects being studied for their own sake, as well as for their applications to different fields of set theory. 

This paper is a first systematic study of \emph{regularity properties} for subsets of generalized Baire spaces. We will focus on regularity properties derived from tree-like forcing partial orders, using the framework introduced by Ikegami in \cite{Ik10} (see Definition \ref{arboreal}) as a generalization of the Baire property, as well as a number of other standard regularity properties (Lebesgue measurability, Ramsey property, Sacks property etc.)  In the classical setting, such  properties have been studied by many people, see, e.g.,  \cite{JSDelta12, BrLo99, BrHaLo, KhomskiiThesis}. Typically, these properties are satisfied by analytic sets, while the Axiom of Choice can be used to provide counterexamples. On the second projective level one obtains independence results, as witnessed by ``Solovay-style'' characterization theorems, such as the following:

\begin{Thm}[Solovay \cite{SolovayAModel}] All $\SIGMA^1_2$ sets have the Baire property if and only if for every $r \in \ww$ there are co-meager many Cohen reals over $L[r]$. \end{Thm}

\begin{Thm}[Judah-Shelah \cite{JSDelta12}] \label{judah} All $\DELTA^1_2$ sets have the Baire property if and only if for every $r \in \ww$ there is a Cohen real over $L[r]$. \end{Thm}

These types of theorems make it possible to study the relationships between different regularity properties on the second level. Far less is known for higher projective levels, although some results exist in the presence of large cardinals (see \cite[Section 5]{Ik10}) and some other results can be found in \cite[Chapter 9]{BaJu95} and in the recent works \cite{FriedmanSchrittesser, CichonPaper}. Solovay's model \cite{SolovayAModel} provides a uniform way of establishing regularity properties for all projective sets, starting from ZFC with an inaccessible.

\bigskip When attempting to generalize descriptive set theory from $\ww$ to $\kk$ for a regular uncountable $\kappa$, at first many basic results remain intact after a straightforward replacement of $\omega$ by $\kappa$. But, before long, one starts to notice fundamental differences: for example, the generalized  $\DELTA^1_1$ sets are not the same as the generalized Borel sets; absoluteness theorems, such as $\Sigma^1_1$- and Shoenfield absoluteness, are not valid; and in the constructible universe $L$, there is a $\SIGMA^1_1$-good well-order of $\kk$, as opposed to merely a $\SIGMA^1_2$-good well-order in the standard setting (see Section \ref{2} for details). Not surprisingly, regularity properties also behave radically different in the generalized context. Halko and Shelah \cite{HalkoShelah} first noticed that on $\dk$, the generalized  Baire property provably fails for $\SIGMA^1_1$ sets. On the other hand, it holds for the generalized Borel sets, and is independent for generalized $\DELTA^1_1$ sets. This suggests that some of the classical theory on the  $\SIGMA^1_2$ and $\DELTA^1_2$ level corresponds to the  $\DELTA^1_1$ level in the generalized setting. 

It should be noted that other kinds of regularity properties have been considered before, sometimes leading to different patterns in terms of consistency of projective regularity. For example, in \cite{SchlichtPSP} Schlicht shows that it is consistent relative to an inaccessible that a version of the perfect set property  holds for all generalized projective sets. By  \cite{LaguzziGeneral}, as well as recent results of Laguzzi and the first author, similar results hold for suitable modifications of the properties studied here.. 

\bigskip This paper is structured as follows:  Section \ref{2}  will  be devoted to a brief survey of facts about the ``generalized reals''. In Section \ref{3} we introduce an abstract notion of regularity and prove that, under certain assumption, the following results hold: \begin{enumerate} 

\item Borel sets are ``regular''.
\item Not all analytic sets are ``regular''.
\item For $\DELTA^1_1$ sets, the answer is independent of ZFC. \end{enumerate} In Section \ref{4} we focus on some concrete examples on the $\DELTA^1_1$-level and generalize some classical results from the $\DELTA^1_2$-level.  Section \ref{5} ends with a number of open questions.

\section{Generalized Baire spaces} \label{2}

We devote this section to a survey of facts about  $\kk$ and $\dk$  which will be needed in the rest of the paper, as well as specifying some definitions and conventions. None of the results here are new, though some are not widely known or have not been sufficiently documented.

\begin{Notation} $\klk$ denotes the set of all functions from $\alpha$ to $\kappa$ for some $\alpha<\kappa$, similarly for $\dlk$. We use standard notation concerning sequences, e.g., for $s,t \in \klk$ we use $s \cc t$ to denote the \emph{concatenation} of $s$ and $t$, $s \subseteq t$ to denote that \emph{$s$ is an initial segment of $t$} etc. $\kuk$ denotes the set of  strictly increasing functions from $\kappa$ to $\kappa$, and $\kulk$ the set of strictly increasing functions from $\alpha$ to $\kappa$ for some $\alpha<\kappa$. Also, we will frequently refer  to elements of $\kk$ or $\dk$ as ``$\kappa$-reals'' or ``generalized reals''.
\end{Notation}

\medskip
\subsection{Topology}  We always assume that $\kappa$ is an uncountable, regular cardinal, and that $\klk = \kappa$ holds.  The \emph{standard topology} on $\kk$   is the one generated by basic open sets of the form $[s] := \{x \in \kk \mid  s \subseteq x\}$, for $s \in \klk$; similarly for  $\dk$. Many elementary facts from the classical setting have straightforward generalizations to the generalized setting. The concepts \emph{nowhere dense} and \emph{meager} are defined as usual, and a set $A$ has the \emph{Baire property} if and only if $A \triangle O$ is meager for some open $O$. The following classical results are true regardless of the value of $\kappa$: \begin{itemize}
\item \emph{Baire category theorem}: the intersection of $\kappa$-many open dense sets is dense.
\item \emph{Kuratowski-Ulam theorem} (also called \emph{Fubini for category}): if $A \subseteq \kk \times \kk$ has the Baire property then $A$ is meager if and only if  $\{x \mid A_x$ is meager$\}$ is comeager, where $A_x := \{y \mid (x,y) \in A\}$.\end{itemize}

 
 \begin{Def} A \emph{tree} is a subset of $\klk$ or $\dlk$  closed under initial segments. For a node $t \in T$, we write $\Succ_T(t) := \{s \in T \mid s = t \cc \left<\alpha\right>$ for some $\alpha\}$. A node $t \in T$ is called \begin{itemize}
  \item \emph{terminal} if $\Succ_T(t) = \varnothing$,
  \item \emph{splitting} if $|\Succ_T(t)| > 1$, and
  \item \emph{club-splitting} if $\{\alpha \mid t \cc \left<\alpha\right> \in T)\}$ is a club in $\kappa$. \end{itemize}
We use the notation $\Split(T)$ to refer to the set of all splitting nodes of $T$.

\p A  $t \in T$ is called a \emph{successor node} if $|t|$ is a successor ordinal and a \emph{limit node} if $|t|$ is a limit ordinal. A tree is \emph{pruned} if it has no terminal nodes, and  \emph{${<}\kappa$-closed} if for every increasing sequence $\{s_i \mid i <\lambda\}$  of  nodes from $T$, for $\lambda<\kappa$, the limit $\bigcup_{i <\lambda}s_i$ is also a node of $T$.  \end{Def}

Notice that concepts such as \emph{club-splitting}, \emph{successor} and \emph{limit} node, and \emph{${<}\kappa$-closed} are inherent to the generalized setting and have no classical counterpart. Most of the trees we consider will be pruned and ${<}\kappa$-closed.

A  \emph{branch through} $T$ is a $\kappa$-real $x \in \kk$ or $\dk$ such that $\forall \alpha \:(x \till \alpha \in T)$, and $[T]$ denotes the set of all branches through $T$. As usual, $[T]$ is topologically closed and every closed set has the form $[T]$ for some (without loss of generality pruned and ${<}\kappa$-closed) tree $T$.

The Borel and projective hierarchies are defined in analogy to the classical situation: the {Borel sets} form the smallest collection of subsets of $\kk$ or $\dk$ containing the basic open sets and closed under complements and $\kappa$-unions. A set is $\SIGMA^1_1$ iff it is the projection of a closed (equivalently: Borel) set; it is $\PI^1_n$ iff its complement is $\SIGMA^1_n$; and it is $\SIGMA^1_{n+1}$ iff it is the projection of a $\PI^1_n$ set, for $n \geq 1$. It is $\DELTA^1_n$ iff it is both $\SIGMA^1_n$ and $\PI^1_n$, and \emph{projective} iff it is $\SIGMA^1_n$ or $\PI^1_n$ for some $n \in \omega$.


In spite of the close similarity of the above notions to the classical ones, there are also fundamental differences:

\begin{Fact} $\mathsf{Borel} \neq \DELTA^1_1$. \end{Fact}

A proof of this fact can be found in \cite[Theorem 18 (1)]{FriedmanHyttinenKulikov}, and we also refer readers to Sections II and III of the same paper for a more detailed survey of the basic properties of $\kk$ and $\dk$.

\medskip
\subsection{The club filter} \label{2club}
Sets that will play a crucial role in this paper are those related to the \emph{club filter}. As usual, we may identify $\dk$ with $\PP(\kappa)$ via characteristic functions.

\begin{Fact}  \label{ClubFilter} The set $C := \{a \subseteq \kappa \mid a$ contains a club$\}$ is $\SIGMA^1_1$. \end{Fact}

\begin{proof} For every $c \subseteq \kappa$, note that $c$ is closed (in the ``club''-sense) if and only if for every $\alpha <\kappa$,  $c \cap \alpha$ is closed in $\alpha$. Therefore, ``being closed'' is a (topologically) closed property. Being unbounded, on the other hand, is a $G_\delta$ property, so ``being club'' is $G_\delta$. Then for all $a \subseteq \kappa$ we have $a \in C$ iff $\exists c \: (c$ is club and $c \subseteq a)$, which is  $\SIGMA^1_1$. \end{proof}


In \cite{HalkoShelah} it was first noticed that the club filter provides a counterexample to the Baire property.

\begin{Thm}[Halko-Shelah] \label{HalkoShelah} The club filter $C$ does not satisfy the Baire property. \end{Thm}

We will prove a generalization of the above, see Theorem \ref{ClubCounterexample}. An immediate corollary of Theorem \ref{HalkoShelah} is that in the generalized setting, analytic sets do not satisfy the Baire property. Although the club filter clearly cannot be Borel (Borel sets \emph{do} satisfy the Baire property, in any topological space satisfying the Baire category theorem), it can consistently be $\DELTA^1_1$ for successors $\kappa$.

\begin{Thm}[Mekler-Shelah; Friedman-Wu-Zdomskyy] \label{ClubDelta} For any successor cardinal $\kappa$, it is consistent that the club filter on $\kappa$ is $\DELTA^1_1$. \end{Thm}

\begin{proof} This was proved for $\kappa = \omega_1$ in  \cite{CanaryTree} and for arbitrary successors $\kappa$ in  \cite{ClubDelta}. \end{proof} 

It is also consistent that the club filter is not $\DELTA^1_1$---this will follow from Theorem \ref{theorem}.

\medskip\subsection{Absoluteness} Two fundamental results in  descriptive set theory are  \emph{analytic} (\emph{Mostowski}) \emph{absoluteness} and \emph{Shoenfield absoluteness}. In general, this type of absoluteness does not hold for uncountable $\kappa$. For example, let $\kappa = \lambda^+$ for regular $\lambda$, pick $S \subseteq \kappa \cap {\rm Cof}(\lambda)$ such that both $S$ and $(\kappa \cap {\rm Cof}(\lambda)) \setminus S$ are stationary. Let $\IP$ be a  forcing for adding a club to $S \cup {\rm Cof}({<}\lambda)$. Then, if $\Phi$ is the $\Sigma^1_1$ formula defining the club filter $C \subseteq \PP(\kappa)$ from Fact \ref{ClubFilter}, we have that $V \models \lnot \Phi(S \cup {\rm Cof}({<}\lambda))$ while $V^\IP \models \Phi(S \cup {\rm Cof}({<}\lambda))$, so $\Sigma^1_1$-absoluteness fails even for $\kappa^+$-preserving forcing extensions. On the other hand, $\Sigma^1_1$-absoluteness does hold for generic extensions via ${<}\kappa$-closed forcings.

\begin{Lem} \label{ClosedAbsoluteness} Let $\IP$ be a ${<}\kappa$-closed forcing. Then $\Sigma^1_1$ formulas are absolute between $V$ and $V^\IP$. \end{Lem}

\begin{proof} Let  $\phi(x)$ be a $\Sigma^1_1$ formula with parameters in $V$. Let $x \in \kk$ and assume $V^\IP \models \phi(x)$. Let $T$ (in $V$) be a two-dimensional tree such that $\{x \mid \phi(x)\} = p[T]$, i.e.,  the projection of $T$ to the first coordinate. Let $h \in \kk \cap V^\IP$ be such that $V^\IP \models (x,h) \in [T]$ and let $\dot{h}$ be a $\IP$-name for $h$.

\p By induction, build an increasing sequence $\{p_i \mid i<\kappa\}$ of $\IP$-conditions, and an increasing sequence $\{t_i \in \klk \mid i<\kappa\}$, such that each  $p_i \Vdash t_i \subseteq \dot{h}$. This can be done since at limit stages $\lambda < \kappa$, we can define $t_\lambda := \bigcup_{i<\lambda} t_i$ and pick $p_\lambda$ below $p_i$ for all $ i<\lambda$. Since every $p_i$ forces $(\check{x}, \dot{h}) \in [T]$, it follows that for every $i$ we have $(x \till |t_i|, t_i) \in T$. But then (in $V$) let $g := \bigcup_{i<\kappa} t_i$, so $(x,g) \in [T]$ and therefore $\phi(x)$ holds.
\end{proof}

\medskip
\subsection{Well-order of the reals} In the classical setting, it is well-known that in $L$ there exists a $\SIGMA^1_2$ well-order of the reals. In fact, the well-order is ``$\SIGMA^1_2$-good'', meaning that both the relation $<_L$ on the reals, and the binary relation defined by $$\Psi(x,y) \equiv  \text{ ``}x \text{ codes the set of $<_L$-predecessors of $y$''}$$ is $\SIGMA^1_2$. The proof uses absoluteness of $<_L$ and $\Psi$ between $L$ and initial segments $L_\delta$ for countable $\delta$, and the fact that ``$E \subseteq \omega \times \omega$ is well-founded'' is a $\PI^1_1$-predicate on $E$. In the generalized setting, however, the predicate ``$E \subseteq \kappa \times \kappa$ is well-founded'' is \emph{closed}, leading to the following result:

\begin{Lem} \label{Wellorder} In $L$, there is a $\SIGMA^1_1$-good well-order of  $\kk$. \end{Lem}

\begin{proof} As usual, we have that for $x,y \in \kk$, $x <_L y$ iff $\exists \delta <\kappa^+$ such that $x,y \in L_\delta$ and $L_\delta \models x <_L y$. Using standard tricks, this can be re-written as ``$\exists E \subseteq \kappa \times \kappa \; (E$ is well-founded, $x,y \in \ran(\pi_E)$ and $(\omega, E) \models ZFC^* + V=L + x <_L y)$'', where $\pi_E$ refers to the transitive collapse of $(\omega,E)$ onto some $(L_\delta,\in)$ and $ZFC^*$ is a sufficiently large fragment of $ZFC$. The statement ``$E$ is well-founded'' is closed because $E$ is well-founded iff $\forall \alpha < \kappa$ $E \cap (\alpha \times \alpha)$ is well-founded. Thus we obtain a $\SIGMA^1_1$ statement. A similar argument works with $<_L$ replaced by $\Psi(x,y)$, showing that the well-order is $\SIGMA^1_1$-good.
\end{proof}

\medskip
\subsection{Proper Forcing} \label{2Forcing} A ubiquitous tool in the study of the classical Baire and Cantor spaces is Shelah's theory of \emph{proper forcing}. It is a technical requirement on a forcing notion which is just sufficient to imply preservation of $\omega_1$, while itself being preserved by countable support iterations, and moreover having a multitude of natural examples. Over the years, there have been various attempts at generalizing this theory to higher cardinals (see e.g. \cite{ShelahNotCollapsing, RoslanowskiShelahMore, FriedmanHonzikZdomskyy} for some recent contributions). Of course, we can use the following straightforward generalization:

\begin{Def} \label{proper} A forcing $\IP$ is \emph{$\kappa$-proper} if for every sufficiently large $\theta$ (e.g. $\theta > 2^{|\IP|}$), and for all elementary submodels $M \prec H_\theta$ such that $|M|=\kappa$ and $M$ is closed under ${<}\kappa$-sequences, for every $p \in \IP \cap M$ there exists $q \leq p$ such that for every dense $D \in M$, $D \cap M$ is predense below $q$. \end{Def}

The above property follows both from the $\kappa^+$-c.c. and a $\kappa$-version of Axiom A, and implies that $\kappa^+$ is preserved, but the property itself is  in general not preserved by iterations, see \cite[Example 2.4]{RoslanowskiSearch}. Nevertheless, it is a useful formulation that we will need on some occasions.

While a uniform theory for $\kappa$-properness is lacking so far,  preservation theorems are usually proved either using the $\kappa^+$-c.c. or on a case-by-case basis.

\begin{Fact} \label{FactusCactus} $\;$ \ \begin{enumerate}

\item $\kappa$-Sacks forcing $\IS_\kappa$ (see Example \ref{ForcingExamples}) was studied by Kanamori \cite{KanamoriSacks}, where the following facts were proved: \begin{enumerate}
  \item $\IS_\kappa$ satisfies a generalized version of Axiom A (see Definition \ref{AxiomA} (2)).
  \item Assuming $\Diamond_\kappa$, iterations of $\IS_\kappa$ with ${\leq} \kappa$-sized supports also satisfy a version of  Axiom A. 
  \item If $\kappa$ is inaccessible, then $\IS_\kappa$ is \emph{$\kk$-bounding} (meaning that for every $x \in \kk \cap V^{\IS_\kappa}$ there exists $y \in \kk \cap V$ such that $x(i) < y(i)$ for sufficiently large $i < \kappa$), and so are arbitrary iterations of $\IS_\kappa$ with ${\leq}\kappa$-size supports.
  
  \end{enumerate}
\item $\kappa$-Miller forcing $\IM_\kappa$ (see Example \ref{ForcingExamples}) was studied by Friedman and Zdomskyy  \cite{FriedmanZdomskyyMiller}, where the following facts were proved: \begin{enumerate}
  \item $\IM_\kappa$ satisfies a generalized version of Axiom A. 
  \item Assuming $\kappa$ is inaccessible, iterations of $\IM_\kappa$ with ${\leq} \kappa$-sized supports satisfy a version of Axiom A. 
\end{enumerate}
\end{enumerate} \end{Fact}
In particular, $\IS_\kappa$, $\IM_\kappa$ and their  iterations  are $\kappa$-proper in the sense of Definition \ref{proper} and thus preserve $\kappa^+$.

\section{Regularity properties} \label{3}

The regularity properties we will consider in this paper are those derived from definable tree-like forcing notions. In this section we give an abstract treatment  following the framework introduced by Ikegami in \cite{Ik10},   providing sufficient conditions so that the following facts can be proved uniformly: \begin{enumerate}
\item Regularity for Borel sets is true.
\item Regularity for arbitrary $\SIGMA^1_1$ sets is false.
\item Regularity for arbitrary $\DELTA^1_1$ sets is independent.
\end{enumerate}

\medskip
\subsection{Tree-like forcings on $\kk$} \label{31}

\begin{Def} \label{arboreal} A forcing notion $\IP$ is called $\kappa$\emph{-tree-like} if the conditions of $\IP$ are pruned and ${<}\kappa$-closed trees on $\kk$ or $\dk$, and for all $T \in \IP$ and all $s \in T$ the restriction $T {\thru} s := \{t \in T \mid s \subseteq t $ or $t \subseteq s\}$ is also a member of $\IP$. The ordering is given by $q \leq p$ iff $q \subseteq p$. Additionally, we require that the property of ``being a $\IP$-tree'' is absolute between models of ZFC. \end{Def}

Below are a few examples of $\kappa$-tree-like forcings that have either been considered in the literature or are natural generalizations of classical  notions.

\begin{Example} \label{ForcingExamples} $\;$ \begin{enumerate}
\item \emph{$\kappa$-Cohen forcing} $\IC_\kappa$. Conditions are the trees  corresponding to the basic open sets $[s]$, for $s \in \dlk$ or $\klk$, ordered by inclusion.

\item \emph{$\kappa$-Sacks forcing} $\IS_\kappa$. A tree $T$ on $\dk$ is called a \emph{$\kappa$-Sacks tree} if it is  pruned, ${<}\kappa$-closed and \begin{enumerate}
\item every node $t \in T$ has a splitting extension in $T$, and 
\item for every increasing sequence $\left<s_i \mid i<\lambda\right>$, $\lambda < \kappa$, of splitting nodes in $T$, $s := \bigcup_{\alpha<\lambda}s_\alpha$ is  a  \emph{splitting} node of $T$. \end{enumerate}
$\IS_\kappa$ is the partial order of $\kappa$-Sacks trees ordered by inclusion. 

\item \emph{$\kappa$-Miller forcing} $\IM_\kappa$. A tree $T$ on $\kulk$ is called a \emph{$\kappa$-Miller tree} if it is pruned, ${<}\kappa$-closed and \begin{enumerate}
\item every node $t \in T$ has a club-splitting extension in $T$,
\item for every increasing sequence $\left<s_i \mid i <\lambda\right>$, $\lambda < \kappa$, of  club-splitting nodes in $T$, $s := \bigcup_{i< \lambda}s_i$ is a club-splitting node of $T$.   Moreover, \emph{continuous club-splitting} is required, which is the following property: for every club-splitting limit node $s \in T$, if $\{s_i \mid i<\lambda\}$ is the set of all club-splitting initial segments of $s$ and $C_i := \{\alpha \mid s_i \cc \left<\alpha\right> \in T\}$ is the club witnessing club-splitting of $s_i$ for every $i$, then $C := \{\alpha \mid s \cc \left<\alpha\right> \in T\} = \bigcap_{i<\lambda} C_i$ is the club witnessing club-splitting of $s$.
\end{enumerate}
$\IM_\kappa$ is the partial order of $\kappa$-Miller trees ordered by inclusion.

\item \emph{$\kappa$-Laver forcing} $\IL_\kappa$. A tree $T$ on $\kulk$ is a $\kappa$\emph{-Laver tree} if all nodes $s \in T$ extending the stem of $T$ are club-splitting. $\IL_\kappa$ is the partial order of $\kappa$-Laver trees ordered by inclusion.

\item \emph{$\kappa$-Mathias forcing} $\IR_\kappa$. A $\kappa$-Mathias condition is a pair $(s,C)$, where $s \subseteq \kappa$, $|s| < \kappa$, $C \subseteq \kappa$ is a club, and $\max(s) < \min(C)$. The conditions are  ordered by $(t,D) \leq (s,C)$ iff $t\leq s$, $D \subseteq C$ and $t \setminus s \subseteq C$. Formally, this does not follow Definition \ref{arboreal}, but we can easily identify conditions $(s,C)$ with trees $T_{(s,C)}$ on $\kulk$ defined by $t \in T_{(s,C)}$ iff $\ran(t) \subseteq s \cup C$.

\item \emph{$\kappa$-Silver forcing} $\IV_\kappa$. If $\kappa$ is inaccessible, let $\IV_\kappa$ consist of $\kappa$-Sacks-trees $T$ on $2^{{<}\kappa}$ which are \emph{uniform}, i.e., for $s,t \in T$, if $|s|=|t|$ then $s \cc \left<i\right> \in T$ iff $t \cc \left<i\right> \in T$. Alternatively, we can view conditions of $\IV_\kappa$ as functions $f: \kappa \to \{0, 1, \{0,1\}\}$, such that $f(i) = \{0,1\}$ holds for all $i \in C$ for some club $C \subseteq \kappa$, ordered by $g \leq f$ iff $\forall i \: (f(i) \in \{0,1\} \to g(i) = f(i))$.
\end{enumerate} \end{Example} 

The generalized $\kappa$-Sacks forcing was introduced and studied by Kanamori in \cite{KanamoriSacks}, and the $\kappa$-Miller forcing is its natural variant, studied e.g. by Friedman and Zdomskyy in \cite{FriedmanZdomskyyMiller}. The reason we require the trees to be ``closed under splitting-nodes'' (2(b) and 3(b)) is to ensure that the resulting forcings are ${<}\kappa$-closed. The property called ``continuous club-splitting'' might seem ad hoc, but it is necessary to show that a version of Axiom A holds for the iteration, see \cite{FriedmanZdomskyyMiller}. We should note that other generalizations of Miller forcing have also been considered, see e.g. \cite{OtherMiller}.

$\kappa$-Silver is a natural generalization of Silver forcing, but   the standard proof of Axiom A only works for inaccessible $\kappa$. 

 $\kappa$-Laver and $\kappa$-Mathias are, again, natural generalizations of their classical counterparts; however, since we require the trees to split into club-many successors at all branches above the stem,  any two $\kappa$-Laver and $\kappa$-Mathias conditions with the same stem are compatible, so both $\IL_\kappa$ and $\IR_\kappa$ are $\kappa^+$-centered and hence satisfy the $\kappa^+$-c.c. Therefore they are perhaps more reminiscent of the classical \emph{Laver-with-filter} and \emph{Mathias-with-filter} forcings on $\ww$, rather than the actual Laver and Mathias forcing posets. Note that if we would drop club-splitting from the definition and only require stationary or $\kappa$-sized splitting instead, we would lose ${<}\kappa$-closure of the forcing.

\begin{Remark}  One notion  conspicuous by its absence from Example \ref{ForcingExamples} is \emph{random forcing}. To date, it is not entirely clear how random forcing should properly be generalized to uncountable $\kappa$. Recently Shelah proposed a  definition for $\kappa$ weakly compact, and other people have attempted to find suitable definitions, for example Laguzzi in \cite[Chapter 3]{LaguzziThesis}. However, a consensus on the correct definition for arbitrary $\kappa$ has not been reached so far, so in this work we choose to avoid random forcing, as well as the concept \emph{null ideal} and \emph{Lebesgue measurability}. \end{Remark}

 The following definition is based on  \cite[Definition 2.6 and Definition 2.8]{Ik10}. Let $\IP$ be a fixed $\kappa$-tree-like forcing.

\begin{Def} Let $A$ be a subset of $\kk$ or $\dk$. Then \begin{enumerate}
\item $A$ is \emph{$\IP$-null} iff $\forall T \in \IP \: \exists S \leq T$ such that $[S] \cap A = \varnothing$. We denote the ideal of $\IP$-null sets by $\N_\IP$
\item $A$ is \emph{$\IP$-meager} iff it is a $\kappa$-union of $\IP$-null sets. We denote the $\kappa$-ideal of $\IP$-meager sets by $\I_\IP$.
\item $A$ is \emph{$\IP$-measurable} iff $\forall T \in \IP \: \exists S \leq T$ such that $[S] \subseteq^* A$ or $[S] \cap A =^* \varnothing$, where $\subseteq^*$ and $=^*$ refers to ``modulo $\I_\IP$''. \end{enumerate} \end{Def}

For a wide class of tree-like forcing notions, the clause ``modulo $\I_\IP$'' can be eliminated from the above definition: see Lemma \ref{lemmings} (2).

\medskip \subsection{Regularity of Borel sets} \label{32}

In $\ww$, it is not hard to prove that if $\IP$ is proper then all analytic sets are $\IP$-measurable, using  forcing-theoretic arguments and absoluteness techniques (see e.g. \cite[Proposition 2.2.3]{KhomskiiThesis}). These methods are generally not available in the generalized setting. However, we would still like to know that, at least, all Borel subsets of $\kk$ are $\IP$-measurable for all reasonable examples of $\IP$.

\begin{Remark} \label{RemarkoBarko} Basic open sets are $\IP$-measurable for all $\IP$. To see this, let $[s]$ be basic open and $T \in \IP$. If $s \in T$ then by Definition \ref{arboreal} $T {\thru} s \in \IP$ and $[T {\thru} s] \subseteq [s]$, otherwise $[T {\thru} s] \cap [s] = \varnothing$. \end{Remark}

Since being $\IP$-measurable is clearly closed under complements, it remains to verify closure under $\kappa$-sized unions and intersections. For that we introduce some definitions that help to simplify the notion of $\IP$-measurability, and moreover will play a crucial role for the rest of this paper.

\begin{Def} \label{AxiomA} Let $\IP$ be a $\kappa$-tree-like forcing notion on $\kk$ or $\dk$. Then we say that:

\begin{enumerate}
\item  $\IP$ is \emph{topological} if $\{[T] \mid T \in \IP\}$ forms a topology base for $\kk$ (i.e., for all $S,T \in \IP$, $[S] \cap [T]$ is either empty or contains $[R]$ for some $R \in \IP$).


\item $\IP$ satisfies \emph{Axiom A} iff there are orderings $\{\leq_\alpha \mid \alpha<\kappa\}$, with $\leq_0 = \leq$, satisfying:
\begin{enumerate}
\item $T \leq_\beta S$ implies $T \leq_\alpha S$, for all $\alpha \leq \beta$.
\item If $\left<T_\alpha \mid \alpha < \lambda\right>$ is a sequence of conditions, with $\lambda \leq \kappa$ (in particular $\lambda = \kappa$) satisfying
$$T_\beta \leq_\alpha T_{\alpha} \text{   for all }\alpha\leq\beta,$$ then there exists $T \in \IP$ such that $T \leq_\alpha T_\alpha$ for all $\alpha<\lambda$.
\item For all $T \in \IP$, $D$ dense below $T$, and $\alpha < \kappa$, there exists an $E \subseteq D$ and $S \leq_\alpha T$ such that $|E| \leq \kappa$ and $E$ is predense below $S$. \end{enumerate}

\item $\IP$ satisfies \emph{Axiom A$^*$} if 2 above holds, but in 2 (c) we additionally require that ``$[S] \subseteq \bigcup \{[T] \mid T \in E\}$''. 
\end{enumerate} \end{Def}

\begin{Example} In Example \ref{ForcingExamples}, $\kappa$-Cohen, $\kappa$-Laver and $\kappa$-Miller are topological. By Fact \ref{FactusCactus},  $\kappa$-Miller and $\kappa$-Sacks satisfy Axiom A, and it is not hard to see that in fact  they satisfy Axiom A$^*$ as well (a direct consequence of the construction). Assuming $\kappa$ is inaccessible, a  generalization of the classical proof shows that $\kappa$-Silver also satisfies Axiom A$^*$. \end{Example}

%
%
%
%
%

\vbox{
\begin{Lem} $\;$ \ \label{lemmings}

\begin{enumerate}
\item If $\IP$ is topological then a set $A$ is $\IP$-measurable iff it satisfies the property of Baire in the topology generated by $\IP$. In particular, all Borel sets are $\IP$-measurable.

\item If $\IP$ satisfies Axiom A$^*$ then $\N_\IP = \I_\IP$, and consequently a set $A$ is $\IP$-measurable iff $\forall T \in \IP \: \exists S \leq T \: ([S] \subseteq A$ or $[S] \cap A = \varnothing)$ $($i.e., we can forget about ``modulo $\I_\IP$''$)$. Moreover, the collection of $\IP$-measurable sets is closed under $\kappa$-unions and $\kappa$-intersections.   \end{enumerate} \end{Lem}
}

The proofs are essentially analogous to the classical situation, but let us present them anyway since they are not widely known. 

\begin{proof} 1. First of all, notice that if $\IP$ is topological then $\N_\IP$ is exactly the collection of nowhere dense sets in the $\IP$-topology and $\I_\IP$ is exactly the ideal of  meager sets in the $\IP$-topology. 

\p First assume $A$ satisfies the $\IP$-Baire property, then let $O$ be an open set in the $\IP$-topology such that $A \triangle O$ is $\IP$-meager. Given any $T \in \IP$, we have two cases: if $[T] \cap O = \varnothing$ then we are done since $[T] \cap A =^* \varnothing$. If $[T] \cap O$ is not empty then there exists a $S \leq T$ such that $[S] \subseteq [T] \cap O$. Then $[S] \subseteq^* A$ holds, so again we are done.

\p The converse direction is somewhat more involved (cf. \cite[Theorem 8.29]{Kechris}). Assume $A$ is $\IP$-measurable. Let \begin{itemize}
\item $D_1$ be a maximal mutually disjoint subfamily of $\{T \in \IP \mid [T] \subseteq^* A\}$,
\item $D_2$ be a maximal mutually disjoint subfamily of $\{T \in \IP \mid [T] \cap A =^* \varnothing\}$, and
\item $D := D_1 \cup D_2$. \end{itemize}
Also write $O_1 := \bigcup\{[T] \mid T \in D_1\}$, $O_2 := \bigcup\{[T] \mid T \in D_2\}$ and $O := O_1 \cup O_2$. We will show that $A \triangle O_1$ is $\IP$-meager.

\s{Claim 1.} $O$ is $\IP$-open dense. 

\begin{proof}[Proof of Claim] Start with any $T$. By assumption there exists $S \leq T$ such that $[S] \subseteq^* A$ or  $[S] \cap A =^* \varnothing$. In the former case, note that by maximality, there must be some $S' \in D_1$ such that $[S] \cap [S'] \neq \varnothing$. Then find $S''$ such that $[S''] \subseteq [S] \cap [S']$. Then $[S''] \subseteq O_1$. Likewise, in the case $[S] \cap A =^* \varnothing$ we find a stronger $S''$ with $[S''] \subseteq O_2$. 
\renewcommand{\qed}{\hfill $\Box$ (Claim 1).} \end{proof}

\s{Claim 2.} \emph{$A \cap O_2$ and $O_1 \setminus A$ are $\IP$-meager.}

\begin{proof}[Proof of Claim] Since the proof of both statements is analogous, we only do the first.

Enumerate $D_2 := \{T_\alpha \mid \alpha < |\kk|\}$. For each $\alpha$, let $\{X^\alpha_i \mid i<\kappa\}$ be a collection of $\IP$-nowhere dense sets, such that $[T_\alpha] \cap A = \bigcup_{i<\kappa} X^\alpha_i$. 
 Now, for every $i < \kappa$, let $Y_i := \bigcup_{\alpha < |\kk|} X^\alpha_i$. We will show that each $Y_i$ is $\IP$-nowhere dense. So fix $i$ and pick any $T \in \IP$: if $[T]$ is disjoint from all $[T_\alpha]$'s then clearly also $[T] \cap Y_i = \varnothing$. Else, let $T_\alpha$ be such that $[T] \cap [T_\alpha] \neq \varnothing$. Then there exists $S \leq T$ such that $[S] \subseteq [T] \cap [T_\alpha]$. By assumption, $[T_\alpha]$ is disjoint from all $[T_\beta]$'s, and hence from all $X^\beta_i$'s, for all $\beta \neq \alpha$. Next, since $X^\alpha_i$ is $\IP$-nowhere dense, we can find $S' \leq S$ such that $[S'] \cap X^\alpha_i = \varnothing$. But then $[S'] \cap Y_i = \varnothing$, proving that $Y_i$ is indeed $\IP$-nowhere dense.
 
\p Now clearly $O_1 \cap A$ is completely covered by the collection $\{Y_i \mid i<\kappa\}$, therefore it is meager. \renewcommand{\qed}{\hfill $\Box$ (Claim 2).} \end{proof}

Now it follows from Claim 1 and Claim 2 that $A \triangle O_1 = (O_1 \setminus A) \cup (A \cap O_2) \cup (A \setminus O	)$ is a union of three meager sets, hence it is meager.

\p This proves that the set $A$ has the property of Baire in the topology generated by $\IP$. 

\p  2. Assume $\IP$ satisfies Axiom A$^*$, and let $\{A_i \mid i<\kappa\}$ be a collection of $\IP$-null sets. We want to show that $A := \bigcup_{i<\kappa}A_i$ is also $\IP$-null. For each $i$ let $D_i := \{T \mid [T] \cap A_i = \varnothing\}$. By assumption, each $D_i$ is dense. Now let $T_0 \in \IP$ be given. Using Axiom A$^*$ find, inductively, a sequence $\{T_i \mid i<\kappa\}$ as well as a sequence $\{ E_i \subseteq D_i \mid i < \kappa\}$ such that \begin{itemize}
\item $T_j \leq_i T_i$ for all $i \leq j$ and 
\item $[T_i] \subseteq \bigcup \{[T] \mid T \in E_i\}$ for all $i$. \end{itemize}
This can always be done by condition (c) of Axiom A$^*$. Then, by condition (b) there is a $T$ such that $T \leq T_i$ for all $i$, and hence, $[T] \subseteq D_i$ for all $i$. In particular, $[T] \cap A_i = \varnothing$ for all $i < \kappa$, proving that $\bigcap A_i$ is $\IP$-null. 

\p For the second claim, it suffices to show closure under $\kappa$-unions. Consider a collection $\{A_i \mid i<\kappa\}$ of $\IP$-measurable sets, and let $T \in \IP$. We must find $S \leq T$ such that $[S] \subseteq \bigcup_{i < \kappa} A_i$ or $[S] \cap \bigcup_{i<\kappa} A_o = \varnothing$. If for at least one $i < \kappa$, we can find $S \leq T$ such that $[S] \subseteq A_i$, we are done. If that's not the case, then notice that each $A_i$ must be in $\N_\IP$, since it is $\IP$-measurable. But by the above this implies $\bigcup_{i<\kappa} A_i \in \N_\IP$, so indeed we can find $S \leq T$ with $[S] \cap \bigcup_{i<\kappa} A_i = \varnothing$. \qedhere  \end{proof}

\begin{Cor} \label{Borel} If $\IP$ is either topological or satisfies Axiom A$^*$ then all Borel sets are $\IP$-measurable. \end{Cor}

\medskip
\subsection{Regularity of $\SIGMA^1_1$ sets} \label{33}

Let us abbreviate ``all sets of complexity $\G$ are $\IP$-measurable'' by ``$\G(\IP)$''. In the $\ww$ case, ZFC proves $\SIGMA^1_1(\IP)$, and by symmetry $\PI^1_1(\IP)$, but $\SIGMA^1_2(\IP)$ and $\DELTA^1_2(\IP)$ are independent of ZFC. But in the case that $\kappa > \omega$  things are dramatically different since by the Halko-Shelah result (Theorem \ref{HalkoShelah})  $\SIGMA^1_1(\IC_\kappa)$ is false, i.e., the Baire property fails for analytic sets. We attempt to find the essential requirements on $\IP$ which would allow us to generalize this proof and show, in ZFC, that $\SIGMA^1_1(\IP)$ fails, i.e., that there is an analytic set which is not $\IP$-measurable. It is most convenient to formulate this requirement in terms of the $\kappa$-Sacks and $\kappa$-Miller forcing notions, see Example \ref{ForcingExamples}.

\begin{Thm} \label{ClubCounterexample} Let $\IP$ be a tree-like forcing notion on $\dk$ whose conditions are $\kappa$-Sacks trees, or a tree-like forcing notion on $\kk$ whose	 conditions are $\kappa$-Miller trees. Then $\SIGMA^1_1(\IP)$ fails. \end{Thm}

\begin{proof} Let's start with the first case. Recall the club-filter $C$ from Fact \ref{ClubFilter}, considered as a subset of $\dk$. If $C$ were $\IP$-measurable then, in particular, we would have a $T \in \IP$ such that $[T] \subseteq^* C$ or $[T] \cap C =^* \varnothing$. First deal with the former case: let $\{X_i \mid i<\kappa\}$ be $\IP$-null sets such that $[T] \setminus C = \bigcup_{i < \kappa} X_i$. Inductively, construct an increasing sequence of splitting nodes in $T$ in such a way that: \begin{itemize}
\item $s_{0} := \stem(T)$,
\item given $s_i$, first extend to $s'_i \in T$ such that $[T{\thru}s'_i] \cap X_i = \varnothing$, then extend further to a splitting node $s_{i+1} \in T$. 
\item at limit stages $\lambda$, note that $s'_\lambda := \bigcup_{i<\lambda}s_i$ is a splitting node by assumption. Let $s_\lambda := s'_\lambda \cc \left<0\right>$.\end{itemize}
Now let $x := \bigcup_{i<\kappa}s_i$. Then $x$ is a branch through $T$, $x \notin X_i$ for all $i$, and moreover, there exists a club $c \subseteq \kappa$ such that $x(i)=0$ for all $i \in c$. In particular, $x \notin C$---contradiction.

\p To deal with the second case that $[T] \cap C =^* \varnothing$, proceed analogously except that at limit stages, pick $s_\lambda := s'_\lambda \cc \left<1\right>$. Then it will follow that $x \in C$.

\p When $\IP$ is a tree-like forcing on $\kk$ whose conditions are $\kappa$-Miller trees, we apply the same argument, but using the following variant of the club-filter: let $S$ be a  stationary, co-stationary subset of $\kappa$ and define
$$C_S := \{a  \in \kk \mid \exists c \subseteq \kappa \text{ club  such that } \forall i \in c \:( x(i) \in S)\}.$$ Clearly this set is $\SIGMA^1_1$ by the same argument as in Fact \ref{ClubFilter}. Proceed exactly as before, choosing members from $S$ or from $\kappa \setminus S$ at limit stages, as desired, which can be achieved using the club-splitting of the trees.
%
%
\end{proof}

 Notice that if we want a $\kappa$-tree-like forcing on $\dk$ to be ${<}\kappa$-closed, it must be a refinement of $\IS_\kappa$, so the above theorem is optimal for ${<}\kappa$-closed tree-like forcings on $\dk$.   For trees on $\kk$, the above theorem is not optimal, although it does seem to take care of many natural examples (for instance those from  Example \ref{ForcingExamples}). 
 A more optimal  version of Theorem \ref{ClubCounterexample} could go according to the following definition: 

\begin{Def} Fix a sequence $\vec{S} := \left<S_i \mid i< \kappa\right>$ of subsets of $\kappa$. For a tree $T$, say that $t \in T$ is $\vec{S}$-\emph{splitting} if $\{\alpha \mid t \cc \left<\alpha\right> \in T\} \cap S_i \neq \varnothing$ and $\{\alpha \mid t \cc \left<\alpha\right> \in T\} \cap (\kappa \setminus S_i) \neq \varnothing$, where $i = |t|$. Say that $T$ is an $\vec{S}$-\emph{splitting tree} if

\begin{itemize}
\item for every $t \in T$ there exists $s \supseteq t$ in $T$ which is $\vec{S}$-splitting, and
\item for every increasing sequence $\left<s_i \mid i<\lambda\right>$ of $\vec{S}$-splitting nodes of $T$, the union $s := \bigcup_{i<\lambda} s_i$ is also an	 $\vec{S}$-splitting node of $T$. \end{itemize}
\end{Def}


\begin{Cor} If $\IP$ is a $\kappa$-tree-like forcing such that, for some sequence $\vec{S}$, every tree in $\IP$ is $\vec{S}$-splitting, then $\SIGMA^1_1(\IP)$ fails. \end{Cor}


In all the above examples, an essential property of the trees $T$ is that $\forall x \in [T]$, the set $\{i < \kappa \mid x \till i$ is a splitting node of $T\}$ forms a club on $\kappa$. Recent work of Philipp Schlicht \cite{SchlichtPSP} and Giorgio Laguzzi \cite{LaguzziGeneral}  suggests that this property is directly related to the existence of $\SIGMA^1_1$-counterexamples, since  for a version of Sacks-, Miller- and Silver-measurability where the trees are \emph{not} required to have this property, it is consistent that all projective sets are measurable. 

\medskip \subsection{Regularity of $\DELTA^1_1$ sets}  \label{34} With $\mathsf{Borel}(\IP)$ being provable in ZFC and  $\SIGMA^1_1(\IP)$ inconsistent, we are  left with the $\DELTA^1_1$-level. 

\begin{Lem}[Folklore] \label{folklore} If $V=L$ then $\DELTA^1_1(\IP)$ is false for  all tree-like $\IP$. \end{Lem}

\begin{proof} Use the $\SIGMA^1_1$-good wellorder of the reals of $L$ from Lemma \ref{Wellorder}, and proceed as in the $\ww$-case, obtaining a $\DELTA^1_1$-counterexample as opposed to a $\DELTA^1_2$ one. \end{proof}

This is not the only method to produce $\DELTA^1_1$-counterexamples to $\IP$-measurability. A completely different method, innate to the generalized setting, is to produce models in which the club filter itself is $\DELTA^1_1$, see Lemma \ref{ClubDelta}.

\bigskip It is known that the Baire property on $\kk$ holds for $\DELTA^1_1$ sets in $\kappa^+$-product/iterations of $\kappa$-Cohen forcing, see e.g.  \cite[Theorem 49 (7)]{FriedmanHyttinenKulikov}. We would like to generalize this to other $\kappa$-tree-like forcings. First, we need the following technical result, a strengthening of the concept of 
\emph{$\kappa$-proper} (Definition \ref{proper}). This is again similar to the classical case.

\begin{Lem} \label{lemmichka} Let $\IP$ be $\kappa$-tree-like, and assume that $\IP$ either has the $\kappa^+$-c.c. or satisfies Axiom A$^*$. Then for every elementary submodel $M \prec \mathcal{H}_\theta$ of a sufficiently large $\mathcal{H}_\theta$, with $|M| = \kappa$ and $M^{{<}\kappa} \subseteq M$, and for every $T \in \IP \cap M$, there is $T' \leq T$ such that
$$[T'] \subseteq^* \{x \in \kk \mid x \text{ is $\IP$-generic over } M\}.$$
where $\subseteq^*$ means ``modulo $\I_\IP$'' and a $\kappa$-real $x$ is $\IP$-generic over $M$ if $\{S \in \IP \cap M \mid x \in [S]\}$ is a $\IP$-generic filter over $M$. \end{Lem}

\begin{proof} First assume that $\IP$ has the $\kappa^+$-c.c. Let $M$ be an elementary submodel with $|M| = \kappa$. 

\s{Claim} \emph{A real $x$ is $\IP$-generic over $M$ if and only if $x \notin B$ for every Borel $\IP$-null set $B$ coded in $M$. }

\begin{proof} Suppose $x$ is $\IP$-generic over $M$, and let $B$ be a $\IP$-null set coded in $M$. 
Then by elementarity $M \models $ ``$B$ is $\IP$-null'', and 
$D := \{S \in \IP \cap M \mid [S] \cap B = \varnothing\}$ is in $M$ and $M \models $ ``$D$ is dense''. Since $x$ is $\IP$-generic, there exists $S \in D$ such that $x \in [S]$, and therefore, $x \notin B$.

\p Conversely, suppose $x \notin B$ for every Borel $\IP$-null set coded in $M$. Let $D \subseteq \IP$ be a dense set in $M$, and let $A$ be a maximal antichain inside $D$. Let $B := \kk \setminus \bigcup \{[S] \mid S \in (A \cap M)\}$ which is a Borel set since $|A| = \kappa$ and has a code in $M$. Moreover $B \in \N_\IP$ since $A$ is a maximal antichain. Therefore, by assumption, $x \notin B$, and hence $x \in [S]$ for some $S	 \in A \cap M$, i.e., $x$ is $\IP$-generic over $M$. \renewcommand{\qed}{\hfill $\Box$ (Claim).} \end{proof}

\p Now it is easy to see that $X := \bigcup\{B \mid B$ is a Borel set in $\N_\IP$ with code in $M\}$ is a $\kappa$-union of $\IP$-null sets, hence it is itself in $\I_\IP$. In particular, there exists $T' \leq T$ such that $[T'] \subseteq^* \{x \mid x $ is $\IP$-generic over $M\} = \kk \setminus  X$.

\medskip \p
Next, assume instead that $\IP$ satisfies Axiom A$^*$. Let $\{D_i \mid i<\kappa\}$ enumerate the dense sets in $M$, and let $T \in \IP \cap M$. As usual, we can apply Axiom A$^*$ to inductively find a fusion sequence $\{T_i \mid i<\kappa\}$ and a sequence $\{E_i \subseteq D_i \mid i< \kappa\}$ such that each $E_i \in M$ and $|E_i| \leq \kappa$, and hence $E_i \subseteq M$, and moreover $[T_i] \subseteq \bigcup\{[S] \mid S \in E_i\}$. Let $T'$ be such that $T' \leq T_i$ for all $i$. Then for every $i$, $[T'] \subseteq \bigcup\{[S] \mid S \in E_i\}$, so, in particular, every $x$ in $[T']$ is $\IP$-generic over $M$, so we are done. \end{proof}

Using this strengthening of $\kappa$-properness, we are almost in a position to prove that a $\kappa^+$-iteration of $\IP$ satisfying either the $\kappa^+$-c.c. or Axiom A$^*$ yields a model of for $\DELTA^1_1(\IP)$. However, we still have an obstacle, and that is the lack of an abstract preservation theorem for $\kappa$-properness, mentioned in Section \ref{2Forcing}. This obstacle makes it impossible to prove the next theorem in an abstract setting including the non-$\kappa^+$-c.c. cases. We could formulate it under the assumption that $\kappa$-properness is preserved; but in fact we only need one consequence of $\kappa$-properness, namely, that all new $\kappa$-reals appear at some initial stage of the iteration (which in particular implies $\kappa^+$-preservation). 

\begin{Thm} \label{theorem} Let $\IP$ be a ${<}\kappa$-closed, $\kappa$-tree-like forcing. \begin{enumerate}
\item Suppose $\IP$ satisfies the $\kappa^+$-c.c., and let $\IP_{\kappa^+}$ be the $\kappa^+$-iteration of $\IP$ with supports of size ${<}\kappa$. Then $V^{\IP_{\kappa^+}} \models \DELTA^1_1(\IP)$. 
\item  Suppose $\IP$ satisfies Axiom A$^*$, and let $\IP_{\kappa^+}$ be the $\kappa^+$-iteration of $\IP$ with supports of size ${\leq} \kappa$. Moreover, assume that for every $x \in \kk \cap V^{\IP_{\kappa^+}}$, there is $\alpha < \kappa^+$ such that $x \in \kk \cap V^{\IP_\alpha}$. Then $V^{\IP_{\kappa^+}} \models \DELTA^1_1(\IP)$. \end{enumerate} \end{Thm}

\begin{proof} The proof works uniformly for both cases. In $V[G_{\kappa^+}]$, let $A$ be $\DELTA^1_1$, defined by $\Sigma^1_1$-formulas $\phi$ and $\psi$. Let $S \in \IP$ be arbitrary. By the assumption, there exists an $\alpha < \kappa^+$ such that all parameters of $\phi$ and $\psi$, as well as $S$, belong to $V[G_{\alpha}]$. Moreover, there is a $\beta > \alpha$ such that $S$ belongs to $G(\beta+1)$ (the $(\beta+1)$-st component of the generic filter), since it is dense to force this for some $\beta > \alpha$. Let $x_{\beta +1}$ be the  real corresponding to $G(\beta+1)$, i.e., the next $\IP$-generic real over $V[G_{\beta}]$. 

\p We know that in the final model $V[G_{\kappa^+}]$, either $\phi(x_{\beta +1})$ or $\psi(x_{\beta +1})$ holds. As $\phi$ and $\psi$ are both $\Sigma^1_1$ the situation is clearly symmetrical so without loss of generality assume the former. Since $\IP$ is ${<}\kappa$-closed, any iteration of it is also ${<}\kappa$-closed, so by Lemma \ref{ClosedAbsoluteness} we have $\Sigma^1_1$-absoluteness between $V[G_{\kappa^+}]$ and $V[G_{\beta+1}]$. In particular,  $V[G_{\beta +1}] = V[G_{\beta}][x_{\beta+1}] \models \phi(x_{\beta +1})$.  By the forcing theorem, and since we have assumed $S \in G(\beta+1)$, there exists a $T \in V[G_\beta]$  such that $T \leq S$ and $T \Vdash_\IP \phi(\dot{x}_\gen)$. 

\p Now, still in $V[G_\beta]$, take an elementary submodel $M$ of a sufficiently large structure, of size $\kappa$, containing $T$. By elementarity, $M \models $ ``$T \Vdash_\IP \phi(\dot{x}_\gen)$''. Going back to $V[G_{\kappa^+}]$, use Lemma \ref{lemmichka} to find a $T' \leq T$  such that $[T'] \subseteq^* \{x \mid x$ is $\IP$-generic over $M\}$. Now note that if $x$ is $\IP$-generic over $M$ and $x \in [T]$, then $M[x] \models \phi(x)$. By upwards-$\Sigma^1_1$-absoluteness  between $M$ and $V[G_{\kappa^+}]$, we conclude that $\phi(x)$ really holds. Since this was true for arbitrary $x \in [T']$, we obtain $[T'] \subseteq^* \{x \mid \phi(x)\} = A$. 
\end{proof}

The above theorem can be applied to many forcing partial orders $\IP$, in particular those from Example \ref{ForcingExamples}.

\begin{Cor}  $\DELTA^1_1(\IP)$ is consistent for $\IP \in \{\IC_\kappa, \IS_\kappa, \IM_\kappa, \IL_\kappa, \IR_\kappa\}$, and if $\kappa$ is inaccessible, also for $\IP =\IV_\kappa$.  \end{Cor}

\begin{proof} Clearly all forcings are ${<}\kappa$-closed. For  $\IC_\kappa, \IL_\kappa$ and $\IR_\kappa$ there are no problems since these forcings have the $\kappa^+$-c.c. By 
Fact \ref{FactusCactus} (1), iterations of $\IS_\kappa$ with ${\leq}\kappa$-sized supports satisfy $\kappa$-properness assuming that $\Diamond_\kappa$ holds in the ground model, so $\DELTA^1_1(\IS_\kappa)$ holds in $L^{\IS_{\kappa^+}}$. By Fact \ref{FactusCactus} (2), iterations of $\IM_\kappa$ with ${\leq}\kappa$-sized supports satisfy $\kappa$-properness for inaccessible $\kappa$. It seems very plausible that by an analogous argument to \cite{KanamoriSacks}, the same holds for arbitrary $\kappa$ assuming $\Diamond_\kappa$. However, we will leave out the verification of this (potentially very technical) proof because $\DELTA^1_1(\IM_\kappa)$ also follows by a much easier argument, namely Theorem \ref{Implications} (3). Finally, if $\kappa$ is inaccessible then a straightforward modification of \cite[Theorem 6.1]{KanamoriSacks} shows that iterations of $\kappa$-Silver with ${\leq}\kappa$-sized supports satisfies $\kappa$-properness (the only change in the argument involves the definition of the fusion sequence \cite[Definition 1.7]{KanamoriSacks} and the  amalgamation defined in \cite[Page 103]{KanamoriSacks}). We leave the details to the reader. \end{proof}

\begin{Remark} \label{RemarkMixing} It is clear that in Theorem \ref{theorem} it is enough to add $\IP$-generic reals cofinally often, provided that the iteration is ${<}\kappa$-closed and satisfies the other requirements. For example, we can obtain $\DELTA^1_1(\IC_\kappa) + \DELTA^1_1(\IL_\kappa) + \DELTA^1_1(\IR_\kappa)$ simultaneously by employing a $\kappa^+$-iteration of $(\IC_\kappa * \IL_\kappa * \IR_\kappa)$ with supports of size ${<}\kappa$. \end{Remark}

\bigskip  Recall that in the classical setting we had Solovay-style characterization theorems for $\DELTA^1_2$ sets, such as Theorem \ref{judah} and related results  (see \cite{BrLo99, Ik10}).  In light of Theorem \ref{theorem}, one might expect that in the generalized setting, analogous characterization theorems exist for statements concerning $\DELTA^1_1$ sets. However, the following observation shows that this is not the case.

\begin{Observation} \label{cannot} Suppose $\kappa$ is successor. There exists a generic extension of $L$ in which the statement ``$\forall r \in \dk\: \exists x  \:(x$ is $\kappa$-Cohen over $L[r])$'' holds, yet there exists a $\DELTA^1_1$ subset of $\dk$ without the Baire property. \end{Observation}

\begin{proof} Recall that by Theorem \ref{ClubDelta}, it is consistent for the club filter  $C$ (Definition \ref{ClubFilter})  to be $\DELTA^1_1$-definable. The idea is to adapt the proof of  \cite[Theorem 1.1]{ClubDelta} due to Friedman, Wu and Zdomskyy. Since that  proof  is long and technical, we cannot afford to go into details here, so we only provide a sketch of the argument and leave the details to the reader. In that proof, a model where $C$ is $\DELTA^1_1$ is obtained by a forcing iteration, starting from $L$, in which cofinally many iterands have the $\kappa^+$-c.c. One can then verify that the proof remains correct if, additionally, $\kappa$-Cohen reals are added cofinally often to this iteration (in fact, $\kappa$-Cohen reals are added naturally in the original proof). Thus we obtain a model in which the club filter is $\DELTA^1_1$ and hence fails to have the Baire property, while clearly the statement ``$\forall r \in \dk\: \exists x  \:(x$ is $\kappa$-Cohen over $L[r])$'' is true. \end{proof}

A similar argument can be applied to any  $\kappa$-tree-like forcing $\IP$ which satisfies the $\kappa^+$-c.c., provided it also satisfies Theorem \ref{ClubCounterexample} (i.e., whose trees are $\kappa$-Sacks or $\kappa$-Miller trees). 

\section{Regularity Properties for $\DELTA^1_1$ sets} \label{4}


In the classical setting, regularity properties related to well-known forcing notions on $\ww$ or $\dw$ have been investigated, and the exact relationship between statements $\DELTA^1_2(\IP)$ and $\SIGMA^1_2(\IP)$ has been studied for various forcing notions $\IP$.   As we saw in the previous section,  for generalized reals the $\DELTA^1_1$-level reflects some of these results. We will focus on the forcing notions from Example \ref{ForcingExamples}, i.e., $\kappa$-Cohen, $\kappa$-Sacks, $\kappa$-Miller, $\kappa$-Laver, $\kappa$-Mathias and $\kappa$-Silver. 

Before proceeding, we make a further comment regarding $\kappa$-Laver and $\kappa$-Mathias, showing that the ideal $\I_{\IL_\kappa}$ of $\IL_\kappa$-meager sets and the ideal $\I_{\IR_\kappa}$ of $\IR_\kappa$-meager sets cannot be neglected when discussing the regularity property generated by them. 

\begin{Lem} \label{nonideal} The ideal $\N_{\IL_\kappa}$ of $\IL_\kappa$-null sets is not equal to the ideal $\I_{\IL_\kappa}$ of $\IL_\kappa$-meager sets. Also, there is an $F_\sigma$ set $A$ such that no $\kappa$-Laver tree is completely contained or completely disjoint from $A$. The same holds for $\IR_\kappa$. \end{Lem} 

\begin{proof}
  Fix a stationary, co-stationary $S \subseteq \kappa$. For each $i <\kappa$ define $A_i := \{x \in \kuk \mid  \forall j > i (x(j) \in S)\}$ and $A = \bigcup_{i<\kappa} A_i$. Then each $A_i$ is $\IL_\kappa$-null, because  any $\kappa$-Laver tree $T$  can be extended to some $T' \leq T$ with stem $s$, such that $|s| > i$ and for some $j>i$ we have $s(j) \notin S$, so that clearly $[T'] \cap A_i = \varnothing$. On the other hand, $A$ itself cannot be $\IL_\kappa$-null, because every $\kappa$-Laver tree $T$ contains a branch $x \in [T]$ such that for all $j$ longer then the stem of $T$ we have $x(j) \in S$, and therefore $x \in A$. It is also clear that the set $A$ is $F_\sigma$ but every $\kappa$-Laver tree $T$ contains a branch $x$ which is in $A$ and another branch $y$ which is not in $A$. The argument for $\kappa$-Mathias is analogous.
\end{proof}

Summarizing, the forcings we have introduced can be  neatly divided into two categories as presented in Table \ref{tablum}.

 \begin{table}[here] \begin{center}
 \begin{tabular}{p{2cm}p{7cm}}
   \hline  
 
  $\kappa$-Cohen  

  $\kappa$-Laver  

  $\kappa$-Mathias  & Category 1: topological, $\kappa^+$-c.c., ideal $\I_\IP$ cannot be neglected; $\IP$-measurability equivalent to Baire property in $\IP$-topology. \\
\hline

  $\kappa$-Sacks  

  $\kappa$-Miller 

  $\kappa$-Silver   & Category 2: non-topological, Axiom A$^*$, $\I_\IP = \N_\IP$ can be neglected. \\ \hline   
 \end{tabular} \end{center} \caption{Properties of forcings.}\label{tablum} \end{table}

\medskip \subsection{Solovay-style characterizations} \label{4Solovay}

By Lemma \ref{cannot}, we  know that a Solovay-style characterization for $\DELTA^1_1(\IP)$ cannot be achieved in the generalized setting. However, in some cases we can obtain one half of such a characterization.

\begin{Lem} $\DELTA^1_1(\IC_\kappa)$ implies that for every $r \in \kk$ there exists a $\kappa$-Cohen real over $L[r]$. \end{Lem}

\begin{proof} The proof is completely analogous to the classical case, see e.g. \cite[Theorem 9.2.1]{BaJu95}, except that we obtain a $\DELTA^1_1$-counterexample as opposed to a $\DELTA^1_2$ one, using the $\SIGMA^1_1$-good wellorder of $L$ (Lemma \ref{Wellorder}). A central ingredient of the classical proof is the Kuratowski-Ulam (Fubini for Category) theorem, which, as we mentioned, is valid on the generalized Baire space. A detailed argument has also been worked out in the PhD Thesis of Laguzzi, see \cite[Theorem 75]{LaguzziThesis}. \end{proof}

\begin{Lem} $\DELTA^1_1(\IS_\kappa)$ implies that for every $r \in \kk$ there is an $x \in \dk \setminus L[r]$. \end{Lem}

\begin{proof} This  follows directly from Lemma \ref{folklore}. \end{proof}

Let us define, for $x,y \in \kk$, the \emph{eventual domination} relation: $x <^* y$ iff $\exists \alpha \forall \beta> \alpha \:( x(\beta)<y(\beta))$. We will simply say ``$y$ dominates $x$'' for $x <^* y$ and if $X \subseteq \kk$ we will say ``$y$ dominates $X$'' iff $\forall x \in X \:(x <^* y)$. We will also say ``$y$ is unbounded over $x$'' iff $x \not>^* y$ and ``$y$ is unbounded over $X$'' iff $\forall x \in X \:( x \not>^* y)$. Note that for the next lemma, it is not relevant whether we talk about domination in the space of all elements of $\kk$ or only the strictly increasing ones.

\begin{Lem} \label{Miller} Suppose $\kappa$ is inaccessible. Then $\DELTA^1_1(\IM_\kappa)$ implies that for every $r \in \kk$ there is an $x \in \kuk$ which is {unbounded} over $\kuk \cap L[r]$. \end{Lem} 

\begin{proof} The proof is based on the proof of \cite[Theorem 6.1]{BrLo99}. Assuming that there are no unbounded reals over $\kuk \cap L[r]$ we will construct a $\SIGMA^1_1$-definable sequence $\left<f_\alpha \mid \alpha < \kappa^+\right>$ of reals in $L[r]$ which is dominating, well-ordered by $<^*$, and satisfies some additional technical properties. This will yield two non-$\kappa$-Miller-measurable sets $A$ and $B$ defined by $A := \{x \in \kuk \mid$ the least $\alpha$ such that $x \leq^* f_\alpha$ is even$\}$ and $B := \{x \in \kuk \mid$the least $\alpha$ such that $x \leq^* f_\alpha$ is odd$\}$, where, by convention, limit ordinals are considered even.

\p To begin with, we fix an enumeration $\left<\sig_i \mid i < \kappa\right>$ of $\kulk \setminus \{\varnothing\}$. 
Let $\quine{\sigma}$ denote $i$ such that $\sigma = \sig_i$, and also well-order $\kulk\setminus \{\varnothing\}$ by $\preceq$, defined by $\sigma \preceq \tau$ iff $\quine{\sigma} \leq \quine{\tau}$. We also use the following notation: for all $\sigma \in \kulk$ of successor length, let $\sigma(\last)$ denote the last digit of $\sigma$, i.e., $\sigma(|\sigma|-1)$.

\p Next, we define a fixed function $\varphi_0: \kulk \to \kappa$ by letting $\varphi_0(\sigma)$ be the least $i < \kappa$ such that $\sig_i(0)>\sigma(\xi)$ for all $\xi<|\sigma|$. Note that since we only consider strictly increasing $\sigma$, this is equivalent to saying ``$\sig_i(0)>\sigma(\last)$'' whenever $|\sigma|$ is successor. The  function $\varphi_0$ should be understood as a ``lower bound'' on possible other functions $\varphi: \kulk \to \kappa$ satisfying $\sig_{\varphi(\sigma)}(0) > \sigma(\xi)$ for all $\xi<|\sigma|$. 

\p Let $T$ be a given $\kappa$-Miller tree $T$, and assume, without loss of generality, that every splitting node of $T$ is club-splitting. We will recursively define a sequence $\left<\tau^T_\sigma \mid \sigma \in \kulk\right>$, another sequence $\left<\tilde{\tau}^T_\sigma \mid \sigma \in \kulk\right>$ consisting of split-nodes of $T$, and a function $\varphi_T: \kulk \to \kappa$. 

\begin{itemize}
 \item $\tau^T_\varnothing = \tilde{\tau}^T_\varnothing = \sig_i$ for the least $i$ such that $\sig_i \in \Split(T)$.
 \item Assuming $\tau^T_\sigma$ and $\tilde{\tau}^T_\sigma$ are defined, and given a $\beta < \kappa$,  
 let $\tau^T_{\sigma \cc \left<\beta\right>}$ be $\sig_i$ for the least $i$ such that \begin{itemize}
 \item $\tilde{\tau}^T_\sigma \cc \sig_i \in \Split(T)$, and
 \item $\sig_i(0) > \beta$. \end{itemize} 
 Then let $\tilde{\tau}^T_{\sigma \cc \left<\beta\right>} := \tilde{\tau}^T_\sigma \cc \tau^T_{\sigma \cc \left<\beta\right>}$.
 \item For $\sigma$ with $|\sigma| = \lambda$ limit, we first let $\hat{\tau}^T_\sigma := \bigcup \{\tilde{\tau}^T_{\sigma \till \alpha} \mid \alpha<\lambda\}$, then pick $\tau_\sigma$ to be $\sig_i$ for the least $i$, such that \begin{itemize}
 \item $\hat{\tau}^T_\sigma \cc \sig_i \in \Split(T)$, and
 \item $\sig_i(0) >  \sigma(\xi)$ for all $\xi<\lambda$ . \end{itemize}
  (Note that this can be done by the assumption that limits of splitting nodes in $T$ are splitting). Then let $\tilde{\tau}^T_{\sigma} := \hat{\tau}^T_\sigma \cc \tau^T_{\sigma}$. 
 \end{itemize}
 Also for every $\sigma$ we define $\varphi_T(\sigma) := \quine{\tau^T_\sigma}$. 

\p Intuitively, each $\tau^T_\sigma$ gives us a $\preceq$-minimal extension within the tree $T$, whose first digit is strictly higher then the a-priori-prescribed values of $\sigma(\xi)$. Then $\tilde{\tau}^T_\sigma$ is the transfinite concatenation $\tau^T_{\sigma \till 0} \cc \tau^T_{\sigma \till 1} \cc \dots \cc \tau^T_{\sigma \till \xi} \cc \dots \cc \tau^T_{\sigma}$ of the previously selected components. The function $\varphi_T$ then gives the corresponding \emph{code}  of $\tau^T_\sigma$, which will be used as a lower bound later. Notice that for any $\kappa$-Miller tree $T$ we have $\varphi_0 \leq \varphi_T$, and in fact $\varphi_0 = \varphi_{\left(\kulk\right)}$ (i.e., the $\varphi_T$ for $T = \kulk = \mathbf{1}_{\IM_\kappa}$).

\p Next, for a fixed function $f: \kappa \to \kappa$, another function $\varphi: \kulk \to \kappa$ satisfying $\varphi_0 \leq \varphi$, and an ordinal $\beta < \kappa$, we define a special, ${<}\kappa$-branching tree $S(\varphi, f, \beta)$. This tree will  be defined as $\bigcup_{\alpha < \kappa} S_\alpha$, where each $S_\alpha$ satisfies the following two requirements: \begin{enumerate}
\item $|S_\alpha| < \kappa$, and    
\item $\exists \rho \in S_\alpha \:(|\rho| \geq \alpha+1)$. \end{enumerate}

\p We construct the $S_\alpha$ recursively as follows:

\begin{itemize}
\item $S_0$ is the tree generated by $\{\sig_i \mid i \leq \beta\}$.
\item $S_1$ is the tree generated by $$\{\rho \cc \sig_i \mid \rho \in S_0, |\rho| \geq 1, i \leq \varphi(\left<\beta\right>) \text{ and } \sig_i(0) > \beta\}.$$ Notice that since $\varphi_0(\left<\beta\right>) \leq \varphi(\left<\beta\right>)$ there is at least one ``new'' $\sig_i$ satisfying the above requirement, and so there is at least one element of $S_1$ of length $\geq 2$. It is also clear that  $|S_1| < \kappa$.

\item Let $\height(S_1) := \sup\{|\rho| \mid \rho \in S_1\}$ and let $f^*(1) := \sup(\{\beta\} \cup \{f(\xi) \mid \xi < \height(S_1)\})$.  Now let $S_2$ be the tree generated by $$\{\rho \cc \sig_i \mid \rho \in S_1, |\rho| \geq 2,  i \leq \varphi(\left<\beta, f^*(1)\right>) \text{ and } \sig_i(0) > f^*(1)\}.$$ Again notice that since $\varphi_0(\left<\beta,f^*(1)\right>) \leq \varphi(\left<\beta,f^*(1)\right>)$, there exists at least one element of $S_2$ of length $\geq 3$. Also it is clear that $|S_2| < \kappa$.

\item Generally, assume $S_\alpha$ is defined, as well as $f^*(\xi)$ for all $\xi < \alpha$. Let $\height(S_\alpha) := \sup\{|\rho| \mid \rho \in S_\alpha\}$, which is an ordinal $<\kappa$ by the inductive assumption that $|S_\alpha| < \kappa$.  Let  $f^*(\alpha) := \sup(\{\beta\} \cup \{f(\xi) \mid \xi < \height(S_\alpha)\})$ and 
let $S_{\alpha+1}$ be the tree generated by
$$\{\rho \cc \sig_i \mid \rho \in S_\alpha, |\rho| \geq \alpha +1, i \leq \varphi(\left<\beta, f^*(1), \dots, f^*(\alpha)\right>) \text{ and } \sig_i(0) > f^*(\alpha)\}.$$ 
As before, $\varphi_0(\left<\beta, f^*(1), \dots, f^*(\alpha)\right>) \leq \varphi(\left<\beta, f^*(1), \dots, f^*(\alpha)\right>)$ implies that $S_{\alpha+1}$ has at least one element of length $\geq \alpha+2$. Also $|S_{\alpha+1}| < \kappa$ is clear.

\item Suppose $\lambda$ is limit. First define $\hat{S}_\lambda$ to be collection of all cofinal branches through $\bigcup_{\alpha<\lambda} S_\alpha$, i.e., 
$$\hat{S}_\lambda :=  \{\rho \in \kulk \mid \forall \xi < {\rho} \: (\rho \till \xi \in \bigcup_{\alpha<\lambda} S_\alpha)\}.$$ 
Since inductively each $S_\alpha$ has branches of length $\geq \alpha+1$ it follows that $\hat{S}_\lambda$ has at least one cofinal branch. 
Moreover, by the inductive assumption that $|S_\alpha| < \kappa$ for all $\alpha$ and the inaccessibility of $\kappa$ it follows that $|\hat{S}_\lambda| <\kappa$.

\p Next, using the notation
$$\vec{f} \till \lambda := \left<\beta\right> \cc \left<f^*(\xi) \mid 1 \leq \xi < \lambda \right>.$$ we let $S_{\lambda}$ be the tree generated by
$$\{\rho \cc \sig_i \mid \rho \in \hat{S}_\lambda,  i \leq \varphi(\vec{f} \till \lambda) \text{ and } \sig_i(0) \geq \sup(\ran(\vec{f} \till \lambda))\}.$$ Since $\varphi_0(\vec{f} \till \lambda) < \varphi(\vec{f} \till \lambda)$ it again follows that $S_\lambda$ has branches of length $\geq \lambda+1$, and $|S_\lambda| < \kappa$ since $|\hat{S}_\lambda| < \kappa$.
 \end{itemize}

 Finally, we set  $S(\varphi, f, \beta) := \bigcup_{\alpha < \kappa} S_\alpha$. The essential properties of $S(\varphi,f,\beta)$ are summarized in the next sublemma:

\begin{SubLem} $\:$ \
 
\begin{enumerate}
\item Every $S(\varphi,f,\beta)$ is bounded by a function $g \in \kk$ $($i.e., $\forall x \in [S(\varphi,f,\beta)] \; \forall i <\kappa \;((x(i) < g(i)))$.
\item Every $x \in [S(\varphi,f,\beta)]$ is cofinally often above $f$ $($i.e., $x \not<^* f)$.
\item For every $\kappa$-Miller tree $T$, $f$ and $\varphi$ satisfying $\varphi_T <^* \varphi$, there exists $\beta<\kappa$ such that $[T] \cap [S(\varphi,f,\beta)] \neq \varnothing$.
\end{enumerate}

\end{SubLem}

\begin{proof} $\;$ 

\begin{enumerate}
\item Since inductively we know that  $|S_\alpha| <\kappa$ for every $\alpha$, in particular each $S_\alpha$ is ${<}\kappa$-branching (i.e., $\forall \rho \in S_\alpha \;(|\Succ_{S_\alpha}(\rho)| <\kappa$). Moreover, by construction all nodes of length $\leq \alpha$ are contained in $S_\alpha$. Therefore, the full tree $S(\varphi,f,\beta)$ is also only ${<}\kappa$-branching. Now, using the inaccessibility of $\kappa$ it is easy to find a function $g$ such that for all $x \in [S(\varphi,f,\beta)] \: \forall i \:(x(i) < g(i))$.

\item By construction, each $S_{\alpha+1}$ contains only those $\rho \cc \sig_i$ where $\sig_i(0) > f^*(\alpha)$. In particular  $\sig_i(0) > f(|\rho|)$. Therefore $x(\xi) > f(\xi)$ happens cofinally often whenever we pick a branch $x$ through $[S(\varphi,f,\beta)]$.

\item This is the main point of the proof. First, note that since $\varphi_T <^* \varphi$, there are only ${<}\kappa$-many $\sigma$ satisfying $\varphi_T(\sigma) \geq \varphi(\sigma)$. In particular, we can pick $\beta < \kappa$ such that \begin{enumerate}
\item $\varphi_T(\varnothing) < \beta$, and 
\item $\varphi_T(\left<\beta\right> \cc \sigma) < \varphi(\left<\beta\right> \cc \sigma)$ holds for \emph{all} $\sigma$. 
\end{enumerate} 
After $\beta$ has been fixed, the tree $S(\varphi,f,\beta)$ is also fixed. In particular, $f^*$ can be computed from $f$ and the rest of the tree, as it was done in the construction of the $S_\alpha$'s. Let $$\vec{f} := \left<\beta\right> \cc \left<f^*(\alpha) \mid 1 \leq \alpha < \kappa\right>.$$ and for all $\alpha < \kappa$ let
$$\rho_{\alpha} := \tilde{\tau}^T_{\vec{f} \till \alpha}.$$
Then $x := \bigcup_{\alpha<\kappa} \rho_\alpha = \bigcup_{\alpha<\kappa}\tilde{\tau}^T_{\vec{f} \till \alpha}$ is a branch through $[T]$. On the other hand, we claim that $\rho_\alpha \in S_\alpha$ for all $\alpha$:

\begin{itemize}
 \item Since $\varphi_T(\varnothing)<\beta$ and $\quine{\rho_0} =  \quine{\tilde{\tau}^T_\varnothing} = \varphi_T(\varnothing)$, by construction   $\rho_0 \in S_0$.
 \item Since $\varphi_T(\left<\beta\right>) < \varphi(\left<\beta\right>)$,  $\quine{\tau^T_{\left<\beta\right>}} = \varphi_T(\left<\beta\right>)$, $\tau^T_{\left<\beta\right>}(0)>\beta$, and 
 $$\rho_1 = \tilde{\tau}^T_{\left<\beta\right>}  = \tilde{\tau}^T_\varnothing \cc \tau^T_{\left<\beta\right>} = \rho_0 \cc \tau^T_{\left<\beta\right>},$$ by construction $\rho_1 \in S_1$.
\item Assume $\rho_\alpha \in  S_\alpha$. Since $\varphi_T(\vec{f} \till (\alpha+1)) < \varphi(\vec{f}\till (\alpha+1))$, $\quine{\tau^T_{\vec{f} \till (\alpha+1)}} = \varphi_T(\vec{f} \till (\alpha+1))$, $\tau^T_{\vec{f} \till (\alpha+1)}(0) > f^*(\alpha)$ and
$$\rho_{\alpha+1} = \tilde{\tau}^T_{\vec{f} \till (\alpha+1)} = \tilde{\tau}^T_{\vec{f} \till \alpha} \cc  {\tau}^T_{\vec{f} \till (\alpha +1)} = \rho_{\alpha} \cc {\tau}^T_{\vec{f} \till (\alpha+1)},$$ by construction  $\rho_{\alpha+1} \in S_{\alpha+1}$.
\item For limits $\lambda$, first let $\hat{\rho}_\lambda := \bigcup_{\alpha<\lambda} \rho_\alpha$, which is the same as the $\hat{\tau}^T_{\vec{f} \till \lambda}$ in the definition of $\tilde{\tau}^T$ at limit stages. Note that  $\hat{\rho} \in \hat{S_\lambda}$. Then, since 
$\varphi_T(\vec{f} \till \lambda) < \varphi(\vec{f}\till \lambda)$, $\quine{\tau^T_{\vec{f} \till \lambda}} = \varphi_T(\vec{f} \till \lambda)$, $\tau^T_{\vec{f} \till \lambda}(0) \geq \sup(\ran(f \till \lambda))$, and
$$\rho_{\lambda} = \tilde{\tau}^T_{\vec{f} \till \lambda} = \hat{\tau}^T_{\vec{f} \till \lambda} \cc  {\tau}^T_{\vec{f} \till \lambda} = \hat{\rho}_{\lambda} \cc {\tau}^T_{\vec{f} \till \lambda},$$ by construction it follows that $\rho_\lambda \in S_\lambda$.
\end{itemize}
So, $\rho_\alpha \in S_\alpha$ for all $\alpha  <\kappa$, hence $x = \bigcup_{\alpha<\kappa}\rho_\alpha \in [S(\varphi,f,\beta)]$.$\;$

 \hfill \qedhere ~(Sublemma)
 \end{enumerate} \end{proof}

\p To complete the proof of the main lemma, assume, towards contradiction, that $\kuk \cap L[r]$ is a dominating set, for some $r$. Construct a sequence $\left<f_\alpha \mid \alpha<\kappa\right>$ of elements of $\kuk \cap L[a]$, and an auxiliary sequence $\left<\varphi_\alpha \mid \alpha < \kappa\right>$ of elements of $(\kulk)^\kappa \cap L[a]$, in such a way that:

\begin{enumerate}
 \item $\left<f_\alpha \mid \alpha<\kappa\right>$ and  $\left<\varphi_\alpha \mid \alpha < \kappa\right>$ are well-ordered by $<^*$,
 \item  $\left<f_\alpha \mid \alpha<\kappa\right>$ is a dominating subset of $\kuk \cap L[a]$ and  $\left<\varphi_\alpha \mid \alpha < \kappa\right>$  is a dominating subset of $(\kulk)^\kappa \cap L[a]$,
 \item all $\varphi_\alpha$ are strictly above $\varphi_0$, 
 \item $f_{\alpha+1}$ dominates $[S(\varphi_\alpha,f_\alpha, \beta)]$ for all $\beta$, and
 \item both sequences have $\SIGMA^1_1$-definitions.
\end{enumerate}

To see that this can be done, at each step $\alpha$ inductively pick the $<_{L[a]}$-least $f_\alpha$ and $\varphi_\alpha$ dominating all the previous functions; to satisfy point 4 above, use Sublemma (1) to dominate each  $[S(\varphi_\alpha,f_\alpha, \beta)]$ by a corresponding function $g_\beta$, and then dominate $\{g_\beta \mid \beta < \kappa\}$ by another $g$.

\p Now, as suggested earlier, define  $A := \{x \in \kuk \mid$ the least $f_\alpha$ which dominates $x$ is even$\}$ and $B := \{x \in \kuk \mid$ the least $f_\alpha$ which dominates $x$ is odd$\}$. Clearly $A \cap B = \varnothing$, and by assumption $A \cup B = \kuk$. Since the sequence of $f_\alpha$'s was $\SIGMA^1_1$-definable, the sets $A$ and $B$ are also $\SIGMA^1_1$-definable, hence they are both $\DELTA^1_1$. To reach a contradiction, let $T$ be a $\kappa$-Miller tree, and we will show that $[T]$ contains an element in $A$ and an element in $B$. Since the sequence $\left<\varphi_\alpha \mid \alpha<\kappa\right>$ is dominating, there exists an $\alpha$ such that for all $\xi\geq \alpha$ we have $\varphi_T <^* \varphi_\xi$. In particular $\varphi_T <^* \varphi_\alpha$ and  $\varphi_T <^* \varphi_{\alpha+1}$. By point 3 of the Sublemma, we can find $\beta$ and $\beta'$ such that
$$[T] \cap [S(\varphi_\alpha,f_\alpha, \beta)] \neq\varnothing, \text{ and }$$
$$[T] \cap [S(\varphi_{\alpha + 1},f_{\alpha + 1}, \beta')] \neq\varnothing.$$
Without loss of generality $\alpha$ is even. Let $y$ be an element of the first set. By point 2 of the Sublemma, $y \not<^* f_\alpha$, and by construction, $y <^* f_\alpha$. Hence $y \in B$. Likewise, let $y'$ be an element of the second set. Then by an analogous argument $y' \not<^*f_{\alpha+1}$ but $y' <^* f_{\alpha+2}$. Hence $y' \in A$. This completes the proof. \end{proof}

\begin{Question} Can Lemma \ref{Miller} be proved without assuming that $\kappa$ is inaccessible? \end{Question}

So far, these are the only generalizations of classical Solovay-style characterizations known to us. The other result due to Brendle and L\"owe linked Laver-measurability with dominating reals. However, that proof does not seem to generalize to the $\kk$-setting  because $\kappa$-Laver-measurability differs from classical Laver-measurability in the sense that the ideal $\I_\IL$ cannot be neglected (see Lemma \ref{nonideal}). Therefore the following is still open:

\begin{Question} Does $\DELTA^1_1(\IL_\kappa)$ imply that for every $r \in \kk$, there is an $x$ which is {dominating} over $L[r]$? \end{Question}

Likewise, currently we do not have suitable Solovay-style consequences of the assumptions $\DELTA^1_1(\IV_\kappa)$ and $\DELTA^1_1(\IR_\kappa)$. In the classical setting, there is a connection between these properties and splitting/unsplit reals.

\begin{Question} Can the hypotheses $\DELTA^1_1(\IV_\kappa)$ and $\DELTA^1_1(\IR_\kappa)$ be linked  to the existence of $($a suitable generalization of$)$ splitting/unsplit reals? \end{Question}

\medskip \subsection{Comparing $\DELTA^1_1(\IP)$}


The next questions we want to ask are: 
for which $\IP$ and $\IQ$ does $\DELTA^1_1(\IP)$ imply $\DELTA^1_1(\IQ)$, and for which $\IP$ and $\IQ$ can we construct models where $\DELTA^1_1(\IP) + \lnot  \DELTA^1_1(\IQ)$ holds? We will  prove several implications for arbitrary pointclasses $\G$ in Lemma \ref{Implications}. Classical counterparts of such implications are well-known but generally much easier to prove, as the uncountable context provides combinatorial challenges not present when $\kappa=\omega$.

 Separating regularity properties is currently very difficult for the following two reasons:
\begin{enumerate}
\item We do not have good Solovay-style characterizations, and
\item We do not have good preservation theorems for forcing iterations. \end{enumerate} We will finish this section with the only example of such a separation result currently known to us.

\begin{Lem} \label{Implications} Let $\G$ be a class of subsets of $\kk$ or $\dk$ closed under continuous preimages $($in particular $\G = \DELTA^1_1)$. Then \begin{enumerate}
\item $\G(\IM_\kappa) \Rightarrow \G(\IS_\kappa)$.       
\item $\G(\IV_\kappa) \Rightarrow \G(\IS_\kappa)$.    
\item $\G(\IC_\kappa) \Rightarrow \G(\IM_\kappa)$.    
\item $\G(\IL_\kappa) \Rightarrow \G(\IM_\kappa)$.  
\item $\G(\IR_\kappa) \Rightarrow \G(\IM_\kappa)$.    
\item If $\kappa$ is inaccessible, then $\G(\IC_\kappa) \Rightarrow \G(\IV_\kappa)$.    
%
\end{enumerate} \end{Lem}                                                                                                                                

\begin{proof} $\;$ \

 \begin{enumerate}
  \item Let $A \subseteq \dk$ be a set in $\G$ and let $T$ be a $\kappa$-Sacks tree. We must find a $\kappa$-Sacks tree below $T$ whose branches are completely contained in or disjoint from $A$. Let $\varphi$ be the natural order-preserving bijection identifying $\dlk$ with $\Split(T)$, and $\varphi^*$ the induced homeomorphism between $\dk$ and $[T]$. Further, fix  a stationary, co-stationary set $S \subseteq \kappa$ and enumerate $S := \{\xi_\alpha \mid \alpha < \kappa\}$ and $\kappa \setminus S := \{\eta_\alpha \mid \alpha < \kappa\}$. Let $\psi$ be a map from $\kulk$ to $\dlk$ defined by: 

\begin{itemize}
\item $\psi(\varnothing)=\varnothing$.
\item $\psi(s \cc \left<\alpha\right>) := \left\{ \begin{array}{cc} \psi(s) \cc \left<1\right> \cc 0^{\beta} \cc \left<1\right> & \text{ if } \alpha \in S \text{ and } \alpha = \xi_\beta \\ 
 \psi(s) \cc \left<0\right> \cc 0^{\beta} \cc \left<1\right> & \text{ if } \alpha \notin S \text{ and } \alpha = \eta_\beta  \end{array} \right.$

where $0^\beta$ denotes a $\beta$-sequence of $0$'s.

\item $\psi(s) := \bigcup_{\alpha<\lambda}\psi(s \till \alpha)$, if $|s| = \lambda$ for a limit ordinal.
\end{itemize}
The function $\psi$ is different from a standard encoding of ordinals  by binary sequences, but it is clear that $\psi$ is bijective, since there is an obvious algorithm to compute $\psi^{-1}(s)$ for any $s \in \dlk$. The reason for using this specific function is that we want $\psi(s)$ to be a splitting node whenever $s$ is a club-splitting node. Clearly,  $\psi$ induces a homeomorphism $\psi^*$ between  $\kuk$ and  $\dk \setminus \IQ$, where we use $\IQ$ to denote the \emph{generalized rationals}, i.e., $\IQ := \{x \in \dk \mid |\{i \mid x(i)=1\}|<\kappa\}$. 

\p Let $A' := (\varphi^* \circ \psi^*)^{-1}[A]$, which is in $\G$ by assumption. By $\G(\IM_\kappa)$ we can find a $\kappa$-Miller tree $R$ such that $[R] \subseteq A'$ or $[R] \cap A' = \varnothing$, w.l.o.g. the former. Let $R' := \{\psi(s) \mid s \in R\}$. First, note that $R'$ is a $\kappa$-Sacks tree: this follows because for any $s \in \Split(R)$  there are $\alpha \in S$ and $\beta \notin S$ such that both $s \cc \left<\alpha\right>$ and $s \cc \left<\beta \right>$ are in $R$, which implies that  both $\psi(s) \cc \left<1\right>$ and $\psi(s) \cc \left<0\right>$ are in $R'$, so $\psi(s) \in \Split(R')$. Moreover, since $\psi^*$ is a homeomorphism, we know that $[R'] \setminus \IQ = (\psi^*)$``$[R] \subseteq (\varphi^*)^{-1}[A]$. But since $\IQ$ is a set of size $\kappa$ we can easily find a refinement $R'' \subseteq R'$, which is still a $\kappa$-Sacks tree and moreover $[R''] \subseteq 
(\psi^*)$``$[R] \subseteq (\varphi^*)^{-1}[A]$. Then $(\varphi^*)$``$[R'']$ generates a $\kappa$-Sacks tree which is completely contained in $[T] \cap A$.

\item Let $A \in \G$ and $T \in \IS_\kappa$ and $\varphi$ and $\varphi^*$ be as above. Then $A' := (\varphi^*)^{-1}[A]$ is in $\G$ so there exists a $\kappa$-Silver tree $S$ such that $[S] \subseteq A$ or $[S] \cap A = \varnothing$. As $S$ is a $\kappa$-Sacks tree, clearly $\varphi$``$S$ generates a $\kappa$-Sacks tree below $T$ whose branches are completely contained in or completely disjoint from $A$.

\item Now let $A \subseteq \kuk$ be in $\G$ and let $T$ be a $\kappa$-Miller tree. By shrinking if necessary, we may assume $T$ to have the property that all splitting nodes are club-splitting. Let $\varphi$ be the natural order-preserving bijection between $\kulk$ and $\Split(T)$, and $\varphi^*$ the induced homeomorphism between $\kuk$ and $[T]$. Let $A' := (\varphi^*)^{-1}[A]$. As $A'$ has the Baire property by $\G(\IC_\kappa)$, let $[s]$ be a basic open set such that $[s] \subseteq^* A'$ or $[s] \cap A' =^* \varnothing$, and without loss of generality assume the former. Let $\{X_i \mid i < \kappa\}$ be nowhere dense sets such that $[s] \setminus A' = \bigcup_{i<\kappa}X_i$. We will inductively construct a $\kappa$-Miller tree $S$ such that $[S] \subseteq  A'$ and $[S] \cap X_i = \varnothing$ for all $i<\kappa$. 

\begin{itemize}
\item Let $S_0$ be the tree generated by $\{s\}$.
\item Suppose $S_i$ has been defined for $i<\kappa$. Let $\Term(S_i)$ be the collection of terminal branches of $S_i$ (i.e., those $\sigma \in S_i$ such that $\Succ_{S_i}(\sigma) = \varnothing)$, and for each $\sigma \in \Term(S_i)$ and $\alpha < \kappa$, let $\tau_{\sigma,\alpha}$ be an extension of $\sigma \cc \left<\alpha\right>$ such that $[\tau_{\sigma,\alpha}] \cap X_i = \varnothing$. Now let $S_{i+1}$ be the tree generated by $\{\tau_{\sigma,\alpha} \mid \sigma \in \Term(S_i)$ and $\alpha<\kappa\}$.
\item For limits $\lambda < \kappa$, let $S_\lambda$ be the tree generated by cofinal branches through $\bigcup_{\alpha<\lambda}S_\alpha$.\end{itemize}
By construction, $S := \bigcup_{i<\kappa} S_i$ is a $\kappa$-Miller tree (all splitting nodes of $S$ are in fact fully splitting). Moreover $[S] \subseteq [s]$ and $[S] \cap X_i = \varnothing$ for all $i<\kappa$. In particular, $[S] \subseteq A'$. But now it follows easily that $\varphi$``$S$ generates a $\kappa$-Miller tree below $T$, whose branches are completely contained in $A$.

\item This follows a similar strategy as above, but using the topology generated by $\IL_\kappa$ instead of the standard topology. Let $A \in \kuk$ be in $\G$, $T \in \IM_\kappa$, $\varphi$ and $\varphi^*$ be as above, and let  $A' := (\varphi^*)^{-1}[A]$. As $A'$ is $\IL_\kappa$-measurable, there is a $\kappa$-Laver tree $R$ such that $[R] \subseteq^* A'$ or $[R] \cap A' =^* \varnothing$, where $\subseteq^*$ and $=^*$ means ``modulo $\I_{\IL_\kappa}$''. Without loss of generality assume the former and let $\{X_i \mid i < \kappa\}$ be in $\N_{\IL_\kappa}$  such that $[R] \setminus A' = \bigcup_{i<\kappa}X_i$. Again we will construct a $\kappa$-Miller tree $S$ such that $[S] \subseteq  A'$ and $[S] \cap X_i = \varnothing$ for all $i<\kappa$. 

 We will need to perform a fusion argument on $\IM_\kappa$, so we introduce some terminology. For a $\kappa$-Miller tree $S$, a node $s \in S$ is called an \emph{$i$-th splitting node} iff $s \in \Split(S)$ and the set $\{j < i \mid s \till j \in \Split(S)\}$ has order-type $i$. $\Split_i(S)$ denotes the set of $i$-th splitting nodes of $S$. The standard fusion for $\IM_\kappa$ (cf. Fact \ref{FactusCactus} (2)) is defined by $S' \leq_i S$ iff $S' \leq S$ and $\Split_i(S') = \Split_i(S)$. We will build a fusion sequence $\{S_i \mid i < \kappa\}$ of $\kappa$-Miller trees, but with the following additional property
 $$(*) \;\;\;\;\;\;\; \forall i \: \forall j \geq i \: \forall s \in \Split_j(S_i) \: (S_i \till s \in \IL_\kappa).$$

\begin{itemize}
\item Let $S_0 := R$. 

\item Suppose $S_i$ has been defined for $i<\kappa$. For each $\sigma \in \bigcup\{\Succ_{S_i}(\rho) \mid \rho \in \Split_i(S_i)\}$, we know by the inductive assumption that $ S_i {\thru} \sigma$ is a $\kappa$-Laver tree. So, let $S_{\sigma} \leq S_i {\thru} \sigma$ be a $\kappa$-Laver tree such that $[S_\sigma] \cap X_i = \varnothing$, and then let $$S_{i+1} := \bigcup \{S_\sigma \mid \sigma \in \bigcup\{\Succ_{S_i}(\rho) \mid \rho \in \Split_i(S_i)\}\}.$$ By construction $S_{i+1}$ is a $\kappa$-Miller tree, $S_{i+1} \leq_i S_i$, and condition $(*)$ is satisfied.

\item For limits $\lambda < \kappa$, let $S_\lambda := \bigcap_{i<\lambda}S_i$. By a standard fusion argument, $S_\lambda$ is a $\kappa$-Miller tree and $S_\lambda \leq_i S_i$ for all $i<\lambda$. Also, given any $\sigma \in \Split_{j}(S_\lambda)$, for any $j \geq \lambda$, by condition $(*)$ we inductively know that $S_i \thru \sigma$ is a $\kappa$-Laver tree for all $i< \lambda$. Then $S_\lambda \thru \sigma = \bigcap_{i<\lambda} (S_i \thru \sigma)$, which is a $\kappa$-Laver tree by the ${<}\kappa$-closure of $\IL_\kappa$. Hence $S_\lambda$ satisfies condition $(*)$.

\end{itemize}
By construction, $S := \bigcap_{i<\kappa} S_i$ is a $\kappa$-Miller tree, $[S] \subseteq [R]$, and $[S] \cap X_i = \varnothing$ for all $i<\kappa$. In particular, $[S] \subseteq A'$. Now it follows that $\varphi$``$S$ generates a $\kappa$-Miller tree below $T$, whose branches are completely contained in $A$.

\item This part is completely analogous to 4. Note that $\kappa$-Mathias conditions are special kinds of $\kappa$-Laver trees, and $\IR_\kappa$ is also ${<}\kappa$-closed. 

\item  Here it is easier to consider $\IC_\kappa$  on $\dk$ as opposed to $\kk$. It is not hard to see that the two properties are equivalent for $\G$. Let $A \subseteq \dk$ be in $\G$, let $T \in \IV_\kappa$, let $\varphi$ be the natural order-preserving bijection between $\dk$ and the splitnodes of $T$, and let $\varphi^*$ be the induced homeomorphism between $\dk$ and $[T]$. Let $A' := (\varphi^*)^{-1}[A]$, and using $\G(\IC_\kappa)$ let $s \in \dlk$ be such that $[s] \subseteq^* A'$ or $[s] \cap A' =^* \varnothing$, without loss of generality the former. Let $X_i$ be nowhere dense such that $[s] \setminus A' = \bigcup_{i<\kappa} X_i$. As before, we will inductively construct a $\kappa$-Silver tree $S$ such that $[S] \subseteq [s]$ and $[S] \cap X_i = \varnothing$ for all $i$. 

\p In this construction, it will be easier to view $\kappa$-Silver conditions as functions from $\kappa$ to $\{0,1,\{0,1\}\}$. We will use the following notation: for $f: \alpha \to \{0,1, \{0,1\}\}$ let $$[f] := \{x \in 2^{\alpha} \mid \forall i \:(f(i)\ \in \{0,1\} \to x(i)=f(i))\}.$$ Notice that if $f: \kappa \to \{0,1,\{0,1\}\}$ and $f(i) = \{0,1\}$ for club-many $i$, then the corresponding $\kappa$-Silver tree can be defined as $S_f := \{\sigma \in \dlk \mid \sigma \in [f \till |\sigma|]\}$, and we have $[S_f] = [f]$. We will construct a function $f$ as the limit of $f_\alpha$'s, defined as follows:

\begin{itemize}
\item $f_0 := s$.
\item Since $X_0$ is nowhere dense, let $\tau_1$ be such that $[s \cc \left<0\right> \cc \tau_1] \cap X_0 = \varnothing$. Then let $\tau_2 \supseteq \tau_1$ be such that $[s \cc \left<0\right> \cc \tau_2] \cap X_0 = \varnothing$. Now set 
$$f_1 := s \cc \left<\{0,1\}\right> \cc \tau_1.$$ Notice that for any $x \in \dk$ extending any $\sigma \in [f_1]$ we have $x \notin X_0$.

\item Suppose $f_i$ is defined for $i<\kappa$. Let $\{\sigma_\alpha \mid \alpha < 2^i\}$ enumerate all sequences in $[f_i \cc \left<\{0,1\}\right>]$ and define $\{\tau_\alpha \mid \alpha<2^i\}$ by induction as follows:
\begin{itemize}
\item $\tau_0 = \varnothing$.
\item If $\tau_\alpha$ is defined let $\tau_{\alpha+1} \supseteq \tau_\alpha$ be such that $[\sigma_\alpha \cc \tau_{\alpha+1}] \cap X_i = \varnothing$.
\item For limits $\lambda$ let $\tau_{\lambda} := \bigcup_{\alpha<\lambda} \tau_\alpha$. 
\end{itemize}
Then define $\tau_{2^i} := \bigcup_{\alpha<2^i}\tau_\alpha$ and notice that $\tau_{2^i} \in 2^{\delta}$ for $\delta<\kappa$ since $\kappa$ was inaccessible. Now let
$$f_{i+1} := f_i \cc \left<\{0,1\}\right> \cc \tau_{2^i}.$$ It is clear that any $x \in \dk$ extending any $\sigma \in [f_{i+1}]$ is not in $X_i$.

\item For $\gamma$ limit, let $f_\gamma := \bigcup_{i<\gamma}f_i$.
\end{itemize}

Finally, we let $f := \bigcup_{i<\kappa} f_i$. By construction  $f(i) = \{0,1\}$ for club-many $i < \kappa$, and clearly every $x \in [f]$ is not in $X_i$ for any $i < \kappa$. Hence  $S_f := \{\sigma \in \dlk \mid \sigma \in [f \till |\sigma|]\}$ is a $\kappa$-Silver tree with $[S_f] \subseteq A'$. Then $\varphi$``$S_f$ generates a $\kappa$-Silver subtree of $T$ which is completely contained in $A$, as had to be shown. \qedhere
 \end{enumerate}
\end{proof}

Focusing on $\G = \DELTA^1_1$, we can summarize the contents of the above results in Figure \ref{Cichon}.\footnote{We arrange the diagram in this particular way in order to be consistent with previous presentations of similar diagrams, e.g. in \cite{CichonPaper}.} Of particular interest are two implications which are present in the classical setting but still seem open in the general setting:

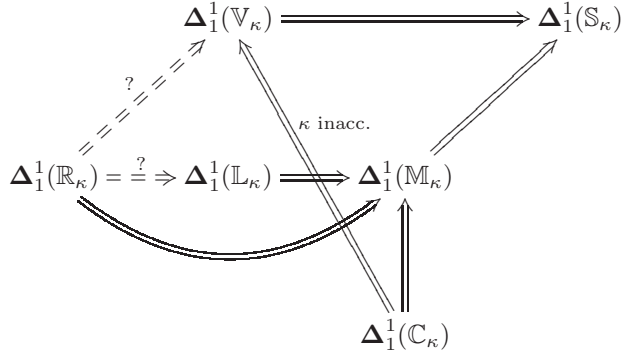
\begin{figure}[here] 
$$
\xymatrix@C=0.9cm@R=1.5cm{
    &  \DELTA^1_1(\IV_\kappa)\ar@{=>}[rr] & & \DELTA^1_1(\IS_\kappa) \\
   \DELTA^1_1(\IR_\kappa)  \ar@{==>}^{?}[r] \ar@{==>}^{?}[ur] \ar@{=>}@/_30pt/[rr]& \DELTA^1_1(\IL_\kappa) \ar@{=>}[r] & \DELTA^1_1(\IM_\kappa) \ar@{=>}[ru] \\
    & & \DELTA^1_1(\IC_\kappa) \ar@{=>}_(.65){\kappa \text{ inacc.}}[luu] \ar@{=>}[u] 
 }
$$\caption{Diagram of implications for $\DELTA^1_1$ .}
\label{Cichon}\end{figure}

\begin{Question} Is $\DELTA^1_1(\IR_\kappa) \Rightarrow \DELTA^1_1(\IL_\kappa)$ true? Is $\DELTA^1_1(\IR_\kappa) \Rightarrow \DELTA^1_1(\IV_\kappa)$ $($at least for $\kappa$ inaccessible$)$ true? \end{Question}
As mentioned, currently we can prove only the following separation theorem.

\begin{Thm} Suppose $\kappa$ is inaccessible. Then it is consistent that $\DELTA^1_1(\IV_\kappa)$ and $\DELTA^1_1(\IS_\kappa)$ hold  whereas $\DELTA^1_1(\IR_\kappa)$, $\DELTA^1_1(\IL_\kappa)$, $\DELTA^1_1(\IC_\kappa)$ and $\DELTA^1_1(\IM_\kappa)$ fail. \end{Thm}

\begin{proof} It is sufficient to establish $\DELTA^1_1(\IV_\kappa) + \lnot \DELTA^1_1(\IM_\kappa)$. Perform a $\kappa^+$-iteration of $\kappa$-Silver forcing, starting in $L$, with supports of size $\kappa$. By Theorem \ref{theorem} we know that $\DELTA^1_1(\IV_\kappa)$ holds in the extension. Also, an argument completely analogous to \cite[Theorem 6.1]{KanamoriSacks} shows that this iteration is  $\kk$-\emph{bounding}, i.e., every function $f \in \kk$ in the extension is dominated by a $g \in \kk$ in the ground model. As a result, the generic extension does not satisfy the statement ``$\forall r \: \exists x \:(x $ is dominating over $\kk \cap L[r])$'', so by Lemma \ref{Miller} $\DELTA^1_1(\IM_\kappa)$ fails. \end{proof}

Notice that by Remark \ref{RemarkMixing} and Lemma \ref{Implications} we can obtain $\DELTA^1_1(\IP)$ for all $\IP \in \{\IC_\kappa, \IS_\kappa, \IM_\kappa, \IL_\kappa, \IR_\kappa\}$, and also for $\IP = \IV_\kappa$ if $\kappa$ is inaccessible, simultaneously in one model, namely $L^{(\IC_\kappa * \IL_\kappa * \IR_\kappa)_{\omega_1}}$.

\section{Open Questions} \label{5}

We have carried out an initial study of regularity properties related to forcing notions on the generalized reals; but many questions remain open, particularly with regard to the specific examples presented in Section \ref{4}.

\begin{Question} $\;$ 

\begin{enumerate} 
\item Can Lemma \ref{Miller} be proved without assuming that $\kappa$ is inaccessible?
\item Does $\DELTA^1_1(\IL_\kappa)$ imply that for every $r \in \kk$, there is an $x$ which is {dominating} over $L[r]$? 
\item  Can the hypotheses $\DELTA^1_1(\IV_\kappa)$ and $\DELTA^1_1(\IR_\kappa)$ be linked  to the existence of $($a suitable generalization of$)$ splitting/unsplit reals? 
\end{enumerate} \end{Question}

A more long-term goal would be to find a complete diagram of implications for generalized $\DELTA^1_1$ sets.
\begin{Question} Which additional implications from Figure \ref{Cichon} can be proved in ZFC? Which are consistently false? Specifically, does $\DELTA^1_1(\IR_\kappa) \Rightarrow \DELTA^1_1(\IL_\kappa)$ and $\DELTA^1_1(\IR_\kappa) \Rightarrow \DELTA^1_1(\IV_\kappa)$ $($at least for $\kappa$ inaccessible$)$ hold?\end{Question}

%
%

In a more conceptual direction, one should try to better understand the exact role of the club filter, which provides counterexamples for $\SIGMA^1_1$-regularity. For example, perhaps one could prove that the club filter, up to some adequate notion of equivalence, is the \emph{only} $\SIGMA^1_1$-counterexample. Alternatively, one could try to focus on regularity properties such as the ones considered in \cite{SchlichtPSP, LaguzziGeneral}, and try to gain a better understanding why the club filter is a counterexample for some regularity properties but not for others. For example, by recent results of Laguzzi and the first author, projective measurability is consistent for a version of Silver forcing  in which the splitting levels occur on a normal measure on $\kappa$ as opposed to the club filter.

\section*{References}

\bibliography{../../10_Bibliography/Khomskii_Master_Bibliography}{}

\end{document}